\documentclass[a4paper]{article}
\usepackage{fullpage}

\usepackage{enumitem}
\usepackage{calc}

\usepackage{lmodern} 

\usepackage{graphicx} 

\usepackage{amsmath, accents}
\usepackage{amssymb,tikz}
\usepackage{pict2e}
\usepackage{amsthm}
\usepackage{amscd}
\usepackage{mathrsfs}
\usepackage{todonotes}

\usepackage[colorlinks=true, pdfstartview=FitV, linkcolor=blue, citecolor=blue, urlcolor=blue]{hyperref}

\usepackage{yfonts}

\usepackage{marvosym}

\newtheorem{theorem}			     {Theorem} [section]
\newtheorem{proposition}[theorem]	 {Proposition}	
	
\newtheorem{lemma}	      [theorem]  {Lemma}		
\theoremstyle{definition}

\newtheorem{remark} {Remark}

\newtheorem{conjecture}[theorem]{Conjecture}

\newcommand{\C}{\mathbb{C}}
\newcommand{\R}{\mathbb{R}}

\newcommand{\Or}{\mathcal{O}}

\newcommand{\re}{\text{\upshape Re\,}}
\newcommand{\im}{\text{\upshape Im\,}}

\newcommand{\Ai}{{\rm Ai}}

\usepackage{tikz}
\usetikzlibrary{arrows}
\usetikzlibrary{decorations.pathmorphing}
\usetikzlibrary{decorations.markings}
\usetikzlibrary{patterns}
\usetikzlibrary{automata}
\usetikzlibrary{positioning}
\usepackage{tikz-cd}

\usepackage{pgfplots}

\numberwithin{equation}{section}

\def\ds{\displaystyle}
\def\bigO{{\cal O}}

\tikzset{->-/.style={decoration={
				markings,
				mark=at position #1 with {\arrow[thick]{>}}},postaction={decorate}}}
	
	\tikzset{-<-/.style={decoration={
				markings,
				mark=at position #1 with {\arrow[thick]{<}}},postaction={decorate}}}
				
\begin{document}
\title{Gap probabilities in the bulk of the Airy process}
\author{Elliot Blackstone, Christophe Charlier, Jonatan Lenells \footnote{Department of Mathematics, KTH Royal Institute of Technology, Lindstedtsv\"{a}gen 25, SE-114 28 Stockholm, Sweden. The work of all three authors is supported by the European Research Council, Grant Agreement No. 682537. Lenells also acknowledges support from the G\"oran Gustafsson Foundation, the Swedish Research Council, Grant No. 2015-05430, and the Ruth and Nils-Erik Stenb\"ack Foundation.
E-mail: elliotb@kth.se, cchar@kth.se, jlenells@kth.se.}}
\maketitle

\begin{abstract}
We consider the probability that no points lie on $g$ large intervals in the bulk of the Airy point process. We make a conjecture for all the terms in the asymptotics up to and including the oscillations of order $1$, and we prove this conjecture for $g=1$.
\end{abstract}
 


\section{Introduction and statement of results}
The Airy point process is a determinantal point process on $\mathbb{R}$ which plays an important role in many seemingly unrelated topics: for example, it models the largest eigenvalues of certain large random matrices \cite{BEY, DeiftGioev,TracyWidom}, the fluctuations at the solid-liquid boundary of certain tiling models \cite{Johansson}, the largest parts of Young diagrams with respect to the Plancherel measure \cite{BaikDeiftRains, BOO}, the fluctuations of the length of the longest increasing subsequence of a random permutation of $\{1,2,\ldots,N\}$, $N \gg 1$ \cite{BaikDeiftJohansson}, and the Brownian particles near the edge-curve in non-intersecting Brownian paths \cite{ADM08}. It has also recently been shown to appear at the liquid-gas boundary of the two-periodic Aztec diamond \cite{BCJ18}. The correlation kernel of the Airy point process is given by
\begin{equation}\label{Airy kernel}
K^\Ai(u, v) = \frac{ \Ai(u) \Ai'(v) - \Ai'(u) \Ai(v) }{ u - v }, \qquad u,v \in \mathbb{R},
\end{equation}
where $\Ai$ denotes the Airy function. There are infinitely many random points which are almost surely all distinct, and there is almost surely a largest point. The Airy process is rigid, in the sense that given a bounded Borel set $A$, the restriction of any point configuration to $\mathbb{R}\setminus A$ almost surely determines the number of points in $A$, see \cite{Bufetov}. We also mention that the maximum fluctuations of the points have been studied in \cite{CorwinGhosal,Zhong} and in \cite[Theorem 1.4]{CharlierClaeys}. 

\medskip

A \textit{gap probability} refers to the probability that a given region is free from points. It is well-known from the general theory of determinantal point processes (see e.g. \cite{Soshnikov}) that gap probabilities are Fredholm determinants. In particular, if $\mathcal{I} \subset \R$ is a finite union of intervals, then the probability of finding no points in $\mathcal{I}$ for the Airy process is given by
\begin{align}\label{generating}
F(\mathcal{I}):=\mathbb{P}(\mbox{no points lie in }\mathcal{I}) = \det(I-\mathcal{K}^{\Ai}\chi_{\mathcal{I}}),
\end{align}
where $\mathcal{K}^{\Ai}$ is the integral operator whose kernel is $K^{\Ai}$, and $\chi_{\mathcal{I}}$ is the projection operator. The properties of $F(\mathcal{I})$ depend on the structure of $\mathcal{I}$: given $g \in \mathbb{N}=\{0,1,2,...\}$ and $\tilde{x}_{0},x_{1},\ldots,x_{2g} \in \mathbb{R}$ such that $x_{2g}<...<x_{1}<\tilde{x}_{0}$, we distinguish the cases
\begin{align*}
& \mathcal{I}_{g}^{(e)} = (x_{2g},x_{2g-1})\cup ... \cup (x_{2},x_{1}) \cup (\tilde{x}_{0},+\infty), \\
& \mathcal{I}_{g}^{(b)} = (x_{2g},x_{2g-1})\cup ... \cup (x_{2},x_{1}).
\end{align*}
If $x_{1}<0$, then $F(\mathcal{I}_{g}^{(b)})$ is the probability of finding no points on a union of $g$ intervals in the bulk, while $F(\mathcal{I}_{g}^{(e)})$ is the probability of finding no points on these $g$ intervals and a gap at the edge. Note that the case $\mathcal{I}_{0}^{(b)} = \emptyset$ leads trivially to $F(\mathcal{I}_{0}^{(b)}) = 1$. The distribution $F(\mathcal{I}_{0}^{(e)}) = F((\tilde{x}_{0},+\infty))$ is the well-known Tracy--Widom distribution of the largest point which is naturally expressed in terms of the solution to a Painlev\'{e} II equation \cite{TracyWidom}. More generally, $F(\mathcal{I}_{g}^{(e)})$ (resp. $F(\mathcal{I}_{g}^{(b)})$) can be expressed in terms of a solution to a system of $2g+1$ (resp. $2g$) coupled Painlev\'{e} II equations \cite{ClaeysDoeraene}, see also \cite{XuDai}. 

\medskip

Gap probabilities are transcendental functions, but in certain cases they possess convenient approximations. For example, for small values of $r>0$, $F(r \mathcal{I})$ represents the probability of a small gap, and can be estimated from its Fredholm series representation. A classical and much harder problem is to find \textit{large gap asymptotics}, which are asymptotics for $F(r \mathcal{I})$ as $r \to + \infty$. Large $r$ asymptotics for $F(r\mathcal{I}_{0}^{(e)})$ and $F(r\mathcal{I}_{1}^{(e)})$ have been analysed in \cite{BBD, DIK, TracyWidom} and \cite{BCL19,KraMarou}, respectively. In this work, we focus on $F(r\mathcal{I}_{g}^{(b)})$ with $g \geq 1$ (the case $g=0$ is trivial, since $F(r\mathcal{I}_{0}^{(b)}) = 1$ for all $r$).

\subsection*{Large gap asymptotics on $g$ intervals in the bulk}
In Conjecture \ref{conj: genus g conjecture} below, we formulate a conjecture for the large $r$ asymptotics of $F(r\mathcal{I}_{g}^{(b)})$ for any finite number of intervals $g \geq 1$, up to and including the oscillations of order $1$. This conjecture has been verified numerically for $g=1,2,3$, and we prove it rigorously for $g=1$.

\medskip

To formulate the conjecture, let $g \geq 1$ be an integer and let $\vec{x} = (x_{1},\ldots,x_{2g}) \in \mathbb{R}^{2g}$ be such that $x_{2g}<...<x_{2}<x_{1}$. Define the square root $\sqrt{\mathcal{R}(z)}$ by
\begin{align*}
& \sqrt{\mathcal{R}(z)}=\prod_{j=0}^{2g}\sqrt{z-x_j},
\end{align*}
where the principal branch is taken for each of the square roots, and $x_{0}=x_{0}(\vec{x})$ is defined below. Let $q_{g+1}=-1$ and define $\{m_{ij}\}_{i,j=1}^g$ and $\{\tilde{m}_i\}_{i=1}^g$ by
\begin{align*}
& m_{ij}=\int_{x_{2i}}^{x_{2i-1}}\frac{s^{j-1}}{\sqrt{\mathcal{R}(s)}}ds, \qquad \tilde{m}_i=-\int_{x_{2i}}^{x_{2i-1}}\frac{q_{g+1}s^{g+1} + q_{g}s^{g}}{\sqrt{\mathcal{R}(s)}}ds, \qquad \text{where} \qquad 
q_{g}=\frac{1}{2}\sum_{j=0}^{2g}x_j.
\end{align*}
The conjecture involves the solution $(x_{0},q_{0},...,q_{g-1}) = (x_{0}(\vec{x}),q_{0}(\vec{x}),...,q_{g-1}(\vec{x})) \in \mathbb{R}^{g+1}$ of the following system of equations:
\begin{align}\label{system for q}
\begin{cases}
\begin{pmatrix} m_{11} & m_{12} & \cdots & m_{1g} \\ m_{21} & m_{22} & \cdots & m_{2g} \\
    \vdots & \vdots & \ddots & \vdots \\ m_{g1} & m_{g2} & \cdots & m_{gg} \end{pmatrix}\begin{pmatrix} q_0 \\ q_1 \\ \vdots \\ q_{g-1} \end{pmatrix}=\begin{pmatrix} \tilde{m}_1 \\ \tilde{m}_2 \\ \vdots \\ \tilde{m}_g \end{pmatrix}, \\[1cm]
\sum_{j=0}^{g+1}q_jx_0^j = 0, \\
x_{0} > \max\{x_{1},0\}.
\end{cases}
\end{align}
Let $\mathcal{A} = \{\vec{x}\in \mathbb{R}^{2g}: x_{2g}<...<x_{2}<x_{1} \mbox{ and there exists } x_{0} \mbox{ satisfying } \eqref{system for q} \}$. For all $g \geq 1$, we expect that $\{\vec{x}\in \mathbb{R}^{2g}: x_{2g}<...<x_{2}<x_{1}<0\}\subset \mathcal{A}$, and we completely characterize $\mathcal{A}$ for $g=1$ in Theorem \ref{thm: main result} below (see also Proposition \ref{prop: x0}). The conjecture also involves the two-sheeted Riemann surface $X$ of genus $g$ associated to $\sqrt{\mathcal{R}(z)}$ where $\sqrt{\mathcal{R}(z)}>0$ for $z\in(x_0,+\infty)$ on the first sheet. We define $A$-cycles and $B$-cycles on $X$ as in Figure \ref{fig:homology} and consider the matrix $\mathbb{A} \in \mathbb{R}^{g\times g}$ given by
\begin{align*}
\mathbb{A}_{jk}:=\oint_{A_j}\frac{s^{k-1}}{\sqrt{\mathcal{R}(s)}}ds = 2 \sum_{l=j}^{g} \int_{x_{2l-1}}^{x_{2l}} \frac{s^{k-1}}{\sqrt{\mathcal{R}(s)}}ds \in \mathbb{R}, \qquad 1 \leq j,k \leq g.
\end{align*}
It follows from the general theory of Riemann surfaces \cite{FarkasKra} that $\mathbb{A}$ is invertible. Consider
\begin{align}\label{def of omegaj}
\begin{pmatrix}
\omega_{1} & \omega_{2} & \cdots & \omega_{g}
\end{pmatrix} = \frac{dz}{\sqrt{\mathcal{R}(z)}} \begin{pmatrix}
1 & z & \cdots & z^{g-1}
\end{pmatrix} \mathbb{A}^{-1}.
\end{align}
By definition, the holomorphic differentials $\omega_k$, $k=1,\dots,g$, satisfy
\begin{align*}
\oint_{A_j}\omega_k =\begin{cases}
1, &j=k, \\
0, &j\neq k,
\end{cases} \qquad 1 \leq j,k \leq g.
\end{align*}
The period matrix $\tau$ defined by
\begin{equation}\label{def of tau genus g}
\tau :=\left[\oint_{B_j}\omega_k\right]_{j,k=1,\dots,g} \in i \, \mathbb{R}^{g \times g}, 
\end{equation}
is symmetric with positive definite imaginary part \cite[page 63]{FarkasKra}. The Riemann theta function $\theta:\mathbb{C}^g\to\mathbb{C}$ associated to $\tau$ is defined by
\begin{equation}\label{def of theta genus g}
\theta(\vec{z};\tau)=\theta(\vec{z}):=\sum_{\vec{n}\in\mathbb{Z}^g}e^{i\pi(\vec{n}^t\tau \vec{n}+2\vec{n}^t\vec{z})}, \qquad \vec{z} \in \mathbb{C}^{g},
\end{equation}
where $^{t}$ denotes the transpose operation. This is an entire function in each component of $\vec{z}$ which satisfies
\begin{align}\label{prop of theta genus g}
& \theta(-\vec{z}) = \theta(\vec{z}) \qquad \mbox{ and } \qquad \theta(\vec{z}+\vec{\lambda}'+\tau\vec{\lambda})=e^{-i\pi(2\vec{\lambda}^t\vec{z}+\vec{\lambda}^t\tau\vec{\lambda})}\theta(\vec{z})
\end{align}
for all $\vec{z}\in \mathbb{C}^{g}$ and $\vec{\lambda}, \vec{\lambda}'\in\mathbb{Z}^g$.
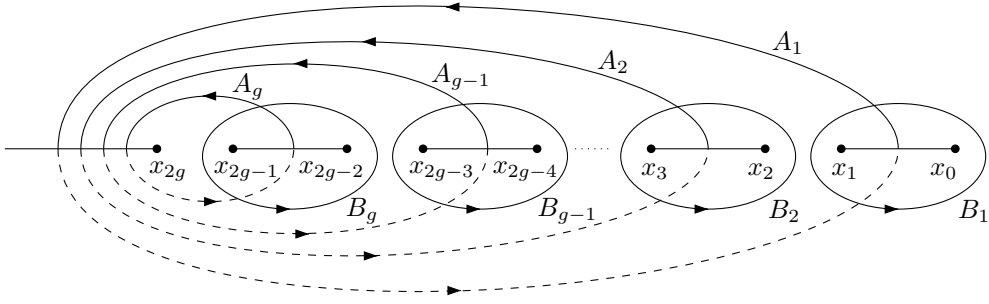
\begin{figure}
\vspace{-2cm}
\centering
\begin{tikzpicture}

\draw (-8,0) -- (-6,0);
\draw (-5,0) -- (-3.5,0);
\draw (-2.5,0) -- (-1,0);
\draw[dotted] (-0.5,0) -- (0,0);
\draw (0.5,0) -- (2,0);
\draw (3,0) -- (4.5,0);

\draw[fill] (-6,0) circle (0.05cm);
\draw[fill] (-5,0) circle (0.05cm);
\draw[fill] (-3.5,0) circle (0.05cm);
\draw[fill] (-2.5,0) circle (0.05cm);
\draw[fill] (-1,0) circle (0.05cm);
\draw[fill] (0.5,0) circle (0.05cm);
\draw[fill] (2,0) circle (0.05cm);
\draw[fill] (3,0) circle (0.05cm);
\draw[fill] (4.5,0) circle (0.05cm);

\node at (-5.85,-0.3) {$x_{2g}$};
\node at (-4.82,-0.3) {$x_{2g-1}$};
\node at (-3.70,-0.3) {$x_{2g-2}$};
\node at (-2.25,-0.3) {$x_{2g-3}$};
\node at (-1.15,-0.3) {$x_{2g-4}$};
\node at (0.55,-0.3) {$x_{3}$};
\node at (1.95,-0.3) {$x_{2}$};
\node at (3.05,-0.3) {$x_{1}$};
\node at (4.35,-0.3) {$x_{0}$};

\draw[black] (-4.25,-0.10) ellipse (1.15cm and 0.70cm);
\draw[black] (-1.75,-0.10) ellipse (1.15cm and 0.70cm);
\draw[black] (1.25,-0.10) ellipse (1.15cm and 0.70cm);
\draw[black] (3.75,-0.10) ellipse (1.15cm and 0.70cm);

\node[black] at (-3.3,-0.82) {$B_g$};
\node[black] at (-0.6,-0.82) {$B_{g-1}$};
\node[black] at (2.25,-0.82) {$B_2$};
\node[black] at (4.75,-0.82) {$B_1$};

\draw[black,arrows={-Triangle[length=0.18cm,width=0.12cm]}]
(-4.25,-0.8) --  ++(0:0.001);
\draw[black,arrows={-Triangle[length=0.18cm,width=0.12cm]}]
(-1.75,-0.8) --  ++(0:0.001);
\draw[black,arrows={-Triangle[length=0.18cm,width=0.12cm]}]
(1.25,-0.8) --  ++(0:0.001);
\draw[black,arrows={-Triangle[length=0.18cm,width=0.12cm]}]
(3.75,-0.8) --  ++(0:0.001);

\draw[black] (-4.2,0) arc(0:180:1.1cm and 0.7cm);
\draw[black,dashed] (-6.4,0) arc(180:360:1.1cm and 0.7cm);
\draw[black] (-6.7,0) .. controls (-6.7,1.5) and (-1.65,1.5) .. (-1.65,0);
\draw[black,dashed] (-6.7,0) .. controls (-6.7,-1.5) and (-1.65,-1.5) .. (-1.65,0);
\draw[black] (-7,0) .. controls (-7,2.25) and (1.25,1.5) .. (1.25,0);
\draw[black,dashed] (-7,0) .. controls (-7,-2.25) and (1.25,-1.5) .. (1.25,0);
\draw[black] (-7.3,0) .. controls (-7.3,3.00) and (3.75,2.0) .. (3.75,0);
\draw[black,dashed] (-7.3,0) .. controls (-7.3,-3.00) and (3.75,-2.0) .. (3.75,0);

\node[black] at (-4.8,0.82) {$A_g$};
\node[black] at (-2,0.95) {$A_{g-1}$};
\node[black] at (0,1.15) {$A_2$};
\node[black] at (2.3,1.4) {$A_1$};

\draw[black,arrows={-Triangle[length=0.18cm,width=0.12cm]}]
(-5.4,0.7) --  ++(0:-0.001);
\draw[black,arrows={-Triangle[length=0.18cm,width=0.12cm]}]
(-5.2,-0.7) --  ++(0:0.001);
\draw[black,arrows={-Triangle[length=0.18cm,width=0.12cm]}]
(-4.2,1.13) --  ++(0:-0.001);
\draw[black,arrows={-Triangle[length=0.18cm,width=0.12cm]}]
(-4,-1.13) --  ++(0:0.001);
\draw[black,arrows={-Triangle[length=0.18cm,width=0.12cm]}]
(-3.3,1.42) --  ++(0:-0.001);
\draw[black,arrows={-Triangle[length=0.18cm,width=0.12cm]}]
(-3.1,-1.42) --  ++(0:0.001);
\draw[black,arrows={-Triangle[length=0.18cm,width=0.12cm]}]
(-2.2,1.88) --  ++(0:-0.001);
\draw[black,arrows={-Triangle[length=0.18cm,width=0.12cm]}]
(-2,-1.88) --  ++(0:0.001);

\end{tikzpicture}
\vspace{-1cm}
\caption{The canonical homology basis on the genus $g$ Riemann surface $X$. The solid parts lie on the first sheet, and the dashed parts lie on the second sheet.}
\label{fig:homology}
\end{figure}

\begin{conjecture}\label{conj: genus g conjecture}
Let $\vec{x} = (x_{1},\ldots,x_{2g}) \in \mathcal{A}$. For $r>0$, define $\vec\nu = \vec{\nu}(r,\vec{x}) = (\nu_{0},\ldots,\nu_{g-1})$ by
\begin{align}
\nu_j=-\frac{\Omega_j}{2\pi}r^\frac{3}{2}, \qquad \mbox{with } \qquad \Omega_j=2i\int_{x_{2j+1}}^{x_{2j}}\frac{\sum_{k=0}^{g+1}q_k\xi^k}{\sqrt{\mathcal{R}(\xi)}_{+}}d\xi>0, \qquad j=0,\dots,g-1. \nonumber
\end{align}
Assume that there exist $\delta_{1},\delta_{2}>0$ such that 
\begin{align}\label{good diophantine prop in thm}
\bigg|\sum_{j=1}^{g}m_{j}\Omega_{j-1}\bigg| \geq \delta_{1}\bigg( \sum_{j=1}^{g}m_{j}^{2}\bigg)^{-\delta_{2}}, \qquad \mbox{for all } m_{1},\ldots,m_{g} \in \mathbb{Z} \mbox{ such that } \sum_{j=1}^{g}m_{j}\Omega_{j-1} \neq 0.
\end{align}
Also suppose that $\Omega_{0},\Omega_{1},\ldots,\Omega_{g-1}$ are rationally independent, that is, if $(n_1, n_2, …, n_g) \in \mathbb{Z}^g$ satisfies
\begin{align}\label{ergo}
\Omega_{0} n_{1}+\Omega_{1} n_{2} + \ldots + \Omega_{g-1} n_{g} = 0,
\end{align}
then $n_{1}=n_{2}= \ldots=n_{g}=0$. As $r \to + \infty$,
\begin{align}
& F(r\vec{x}) = \exp \Big( c\, r^{3} - \frac{3g}{8} \log r + \log \theta(\vec\nu) + C + \bigO ( r^{-\frac{3}{2}} ) \Big), \label{F asymp genus g} 
\end{align}
where $C=C(\vec{x})$ is independent of $r$ and $c=c(\vec{x})$ is given by
\begin{align}
& c = \frac{1}{12}\left[\sum_{j=0}^{2g}x_j^{3}-\sum_{0\leq j<k\leq2g}(x_j^2x_k+x_jx_k^2)-2\sum_{0\leq j<k<l\leq2g}x_jx_kx_l\right]-\frac{2}{3}q_{g-2}-\frac{q_{g-1}}{3}\sum_{j=0}^{2g}x_j. \label{def of C1 genus g}
\end{align}
Furthermore, $\{\vec{x}\in \mathbb{R}^{2g}: x_{2g}<...<x_{2}<x_{1}<0\}\subset \mathcal{A}$.
\end{conjecture}
\begin{remark}
If condition \eqref{good diophantine prop in thm} fails, then the error term $\bigO(r^{-\frac{3}{2}})$ in \eqref{F asymp genus g} must be replaced by $o(\log r)$ as in \cite{DIZ}. On the other hand, if condition \eqref{ergo} does not hold, from \cite{DIZ} and \cite[Theorems 1.1--1.4]{BCL20Bessel} we expect that the log-coefficient $-\frac{3g}{8}$ should be replaced by a rather complicated hyperelliptic integral. Again from \cite{DIZ,BCL20Bessel}, we expect the map $(x_{1},\ldots,x_{2g}) \to (\Omega_{0},\ldots,\Omega_{g-1})$ to contain open balls in $(0,+\infty)^{g}$. Since the set $\{(\Omega_{0},\ldots,\Omega_{g-1})\} \subset (0,+\infty)^{g}$ for which either \eqref{good diophantine prop in thm} or \eqref{ergo} does not hold has Lebesgue measure zero, the asymptotics \eqref{F asymp genus g} are expected to hold for generic $(x_{1},\ldots,x_{2g}) \in \mathcal{A}$. Note that for $g=1$, both \eqref{good diophantine prop in thm} and \eqref{ergo} are satisfied for \textit{all} $(x_{1},x_{2}) \in \mathcal{A}$.
\end{remark}
\begin{remark}
The fact that the oscillations are described in terms of the Riemann $\theta$-function related to a genus $g$ Riemann surface is not surprising; this happens also in the case of the sine process \cite{DIZ, FK20, Widom1995}. 
\end{remark}
\begin{remark}
The above conjecture has been verified numerically for $g=1,2,3$ for several choices of $\vec{x}$ by using the Bornemann Linear Algebra package \cite{Bornemann}.
\end{remark}
\begin{theorem}\label{thm: main result}
Conjecture \ref{conj: genus g conjecture} holds for $g=1$. Furthermore, if $g=1$, then
\begin{align*}
\mathcal{A} = \{(x_{1},x_{2})\in \mathbb{R}^{2}: x_{2}<x_{1}<0 \} \cup \{(x_{1},x_{2})\in \mathbb{R}^{2}: x_{1} \geq 0 \mbox{ and } x_{2} < -2x_{1} \}.
\end{align*}
\end{theorem}


\paragraph{Outline of the paper.} In Section \ref{section:diffid}, we use the method developed by Its, Izergin, Korepin and Slavnov \cite{IIKS}, combined with a result of Claeys and Doeraene \cite{ClaeysDoeraene}, to express $\partial_{r} \log F(r \mathcal{I}_{1}^{\smash {(b)}}) = \partial_{r} \log F(r (x_{2},x_{1}))$ in terms of the solution, denoted $\Psi$, of a $2\times2$ matrix Riemann-Hilbert (RH) problem. In Sections \ref{section:RH1}-\ref{section:smallnorm}, we obtain the asymptotics of $\Psi$ as $r\to+\infty$ via the Deift-Zhou steepest descent method \cite{DKMVZ1,DeiftZhou}. As a consequence of our RH analysis, we obtain large $r$ asymptotics for $\partial_{r} \log F(r (x_{2},x_{1}))$. The simplification of these asymptotics is the most challenging part of this work; this is done in Section \ref{section:integration1}. Finally, we prove Theorem \ref{thm: main result} by integrating the asymptotics of $\partial_{r} \log F(r (x_{2},x_{1}))$ with respect to $r$. 

\section{Differential identity for $F$}\label{section:diffid}
The method of Its, Izergin, Korepin and Slavnov \cite{IIKS} applies to kernels of the \textit{integrable form} $K(x,y)=\frac{\vec{f}^{\hspace{0.6mm}t}(x)\vec{h}(y)}{x-y}$, where $\vec{f}(x)$ and $\vec{h}(y)$ are column vectors satisfying $\vec{f}^{\hspace{0.6mm}t}(x)\vec{h}(x)=0$. Given a subset $A \subset \mathbb{R}$, we write $\chi_{A}$ for the characteristic function of $A$. It is easy to see that
\begin{align*}
K^{\Ai}\big|_{r(x_{2},x_{1})}(x,y) = K^{\Ai}(x,y)\chi_{r(x_{2},x_{1})}(y), \qquad x,y \in \mathbb{R},
\end{align*}
is integrable with $\vec f$ and $\vec h$ given by
\begin{align*}
\vec{f}(x)=\begin{pmatrix}\Ai(x)\\
\Ai'(x)\end{pmatrix}, \qquad \vec{h}(y)=\begin{pmatrix}
\Ai'(y)\chi_{r(x_2, x_1)}(y) \\
-\Ai(y)\chi_{r(x_2, x_1)}(y)
\end{pmatrix}.
\end{align*}
The associated integral operator $\mathcal K_{r}$, acting on $L^{2}(rx_{2},+\infty)$, is given by
\begin{equation}\label{eq:defoperator}
\mathcal{K}_{r}\phi(x)=\int_{rx_{2}}^{+\infty}K^\Ai(x,y)\chi_{r(x_2, x_1)}(y)\phi(y)dy, \qquad \phi \in L^{2}(rx_{2},+\infty).
\end{equation}
By definition, $F(r(x_{2},x_{1})) =\det\left(I-\mathcal{K}_{r}\right) = \mathbb{P}(\mbox{no points lie in }(rx_{2},rx_{1}))>0 $. Hence, using standard properties of trace class operators, we obtain
\begin{align}
\partial_r\log\det\left(I-\mathcal{K}_{r}\right)&=-\text{Tr}\left[\left(I-\mathcal{K}_{r}\right)^{-1}\partial_r\mathcal{K}_{r}\right]=\sum_{j=1}^{2} (-1)^jx_j\text{Tr}\left[\left(I-\mathcal{K}_{r}\right)^{-1}\mathcal{K}_{r} \delta_{rx_j}\right] \nonumber \\
    &=\sum_{j=1}^2(-1)^{j}x_j\text{Tr}\left[\mathcal{R}_{r} \delta_{rx_j}\right] =\sum_{j=1}^2(-1)^{j}x_j\lim_{u\to rx_j}R_{r}(u,u), \label{F_r identity}
\end{align}
where the limits $u\to rx_j$, $j = 1,2$, are taken from the interior of $(rx_2,rx_1)$, $\delta_{rx_j}$ is the Dirac delta operator, the integral operator $\mathcal{R}_{r}$ is given by
\begin{align*}
\mathcal{R}_{r} := \left(I-\mathcal{K}_{r}\right)^{-1}\mathcal{K}_{r} = \left(I-\mathcal{K}_{r}\right)^{-1} - I,
\end{align*}
and $R_{r}$ is the kernel of $\mathcal{R}_{r}$. 
\begin{figure}
\centering
\begin{tikzpicture}
\draw[fill] (0,0) circle (0.05);
\draw (3,0) -- (8,0);
\draw (0,0) -- (120:3);
\draw (0,0) -- (-120:3);
\draw (0,0) -- (-3,0);
\draw[fill] (3,0) circle (0.05);

\node at (0.2,-0.3) {$x_2$};
\node at (3,-0.3) {$x_1$};
\node at (8,-0.3) {$+\infty$};

\node at (98:2) {$\begin{pmatrix} 1 & 0 \\ 1 & 1 \end{pmatrix}$};
\node at (160:2) {$\begin{pmatrix} 0 & 1 \\ -1 & 0 \end{pmatrix}$};
\node at (-98:2) {$\begin{pmatrix} 1 & 0 \\ 1 & 1 \end{pmatrix}$};

\node at (5.5,0.6) {$\begin{pmatrix} 1 & 1 \\ 0 & 1 \end{pmatrix}$};

\draw[black,arrows={-Triangle[length=0.18cm,width=0.12cm]}]
(-120:1.5) --  ++(60:0.001);
\draw[black,arrows={-Triangle[length=0.18cm,width=0.12cm]}]
(120:1.3) --  ++(-60:0.001);
\draw[black,arrows={-Triangle[length=0.18cm,width=0.12cm]}]
(180:1.5) --  ++(0:0.001);

\draw[black,arrows={-Triangle[length=0.18cm,width=0.12cm]}]
(0:5.5) --  ++(0:0.001);

\end{tikzpicture}
\caption{The jump contour $\Gamma$ for the RH problem for $\Psi$.}
\label{fig:modelRHcontours}
\end{figure}
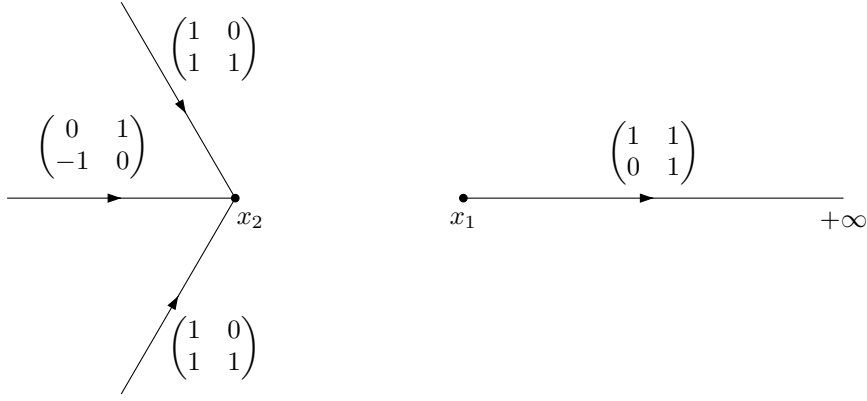
Now, we invoke a result of Claeys and Doeraene \cite{ClaeysDoeraene}: for $u \in (rx_2,rx_1)$, we have
\begin{equation}\label{link between resolvant and RHP}
R_{r}(u,u)=\frac{1}{2\pi i}\left(\Psi_+^{-1}\Psi_+'\right)_{21}(u; rx_{1},rx_{2}),
\end{equation}
where $\Psi$ is the solution to the following RH problem. 
\subsubsection*{RH problem for $\Psi=\Psi(\cdot;x_{1},x_{2})$}
\begin{enumerate}[label={(\alph*)}]
\item[(a)] $\Psi : \C \setminus \Gamma \rightarrow \C^{2\times 2}$ is analytic, where
\begin{equation}\label{eq:defGamma}
\Gamma=e^{\pm\frac{2\pi i}{3}} (x_2,+\infty)\cup (-\infty,x_2] \cup [x_1,+\infty)
\end{equation}
is oriented as in Figure \ref{fig:modelRHcontours}.
\item[(b)] $\Psi(z)$ has continuous boundary values as $\Gamma\backslash \{x_{1},x_{2}\}$ is approached from the left ($+$ side) and from the right ($-$ side) and they are related by
\begin{equation*}
\begin{array}{ll}
\Psi_+(z) = \Psi_-(z) \begin{pmatrix} 1 & 0 \\ 1 & 1 \end{pmatrix} & \textrm{for } z \in e^{\pm\frac{2\pi i}{3}} (x_2,+\infty), \\
\Psi_+(z) = \Psi_-(z) \begin{pmatrix} 0 & 1 \\ -1 & 0 \end{pmatrix} & \textrm{for } z \in (-\infty,x_2), \\
\Psi_+(z) = \Psi_-(z) \begin{pmatrix} 1 & 1 \\ 0 & 1 \end{pmatrix} & \textrm{for } z \in (x_1,+\infty).
\end{array}
\end{equation*}
\item[(c)] As $z \rightarrow \infty$, we have
\begin{equation}
\label{eq:psiasympinf}
\Psi(z) = \left( I + \bigO( z^{-1})\right) z^{\frac{1}{4} \sigma_3} M^{-1} e^{-\frac{2}{3}z^\frac{3}{2}\sigma_3},
\end{equation}
where principal branches of $z^\frac{3}{2}$ and $z^{\frac{1}{4}}$ are taken, and
\begin{align*}
& M = \frac{1}{\sqrt{2}}\begin{pmatrix}
1 & i \\
i & 1
\end{pmatrix}, \qquad \sigma_3 = \begin{pmatrix} 1 & 0 \\ 0 & -1 \end{pmatrix}.
\end{align*}
\item[(d)] $\Psi(z) = \Or( \log(z-x_j) )$ as $z \rightarrow x_j$, $j = 1,2$.
\end{enumerate}
By combining \eqref{F_r identity} with \eqref{link between resolvant and RHP}, we arrive at
\begin{align}\label{final formula diff identity}
\partial_{r} \log F(r(x_{2},x_{1}))&=\sum_{j=1}^2\frac{(-1)^{j}x_j}{2\pi i}\lim_{z \to x_j}\left(\Psi_+^{-1}\Psi_+'\right)_{21}(rz;rx_{1},rx_{2}),
\end{align}
where the limits as $z\to x_j$, $j=1,2$ are taken with $z \in (x_2,x_1)$. 

\section{Asymptotic analysis of $\Psi$: first steps}\label{section:RH1}
The goal of this section is to implement the first steps of the steepest descent method for the RH problem for $\Psi$. In Subsection \ref{subsection: g-function}, we construct the $\mathfrak{g}$-function and list its important properties.  In Subsection \ref{Subsection: T to S}, we use $\mathfrak{g}$ to normalize the RH problem at $\infty$, and then proceed with the so-called opening of the lenses. We henceforth assume that $g=1$.

\subsection{$\mathfrak{g}$-function}\label{subsection: g-function}
Our first goal is to prove that the system of equations (\ref{system for q}) has a unique solution $(x_{0},q_{0}) \in \R^+ \times \mathbb{R}$, where $\R^+ = (0,+\infty)$. According to the last equation in (\ref{system for q}), the degree two polynomial $q(z) := \sum_{j=0}^{g+1}q_j z^j = q_2 z^2 + q_1 z + q_0$ satisfies $q(x_{0})=0$, and is therefore of the form
\begin{align}\label{qPolynomial}
q(z)=-z^{2}+\frac{z}{2}(x_0+x_1+x_2)+q_0, \qquad q_0=\frac{x_0}{2}(x_0-x_1-x_2),
\end{align}
for a certain $x_{0}$ that will be determined so that the first equation in \eqref{system for q} holds, i.e., so that $m_{11}q_0 = \tilde{m}_1$. Defining the function $\mathcal{F}:[x_{1},\infty)\to\mathbb{R}$ by
\begin{equation}\label{Fdef}
\mathcal{F}(x):=\int_{x_2}^{x_1}\frac{\sqrt{x-s}\left(s+\frac{1}{2}(x-x_1-x_2)\right)}{\sqrt{(x_1-s)(s-x_2)}}ds,
\end{equation}
the equation $m_{11}q_0 = \tilde{m}_1$ can be rewritten as $\mathcal{F}(x_0) = 0$. Hence the following proposition implies that (\ref{system for q}) has a unique solution $(x_{0},q_{0}) \in \R^+ \times \mathbb{R}$.

\begin{proposition}\label{prop: x0}
\, 
\begin{itemize}
\item[(a)] Given $x_2<x_1<0$, there is a unique $x_0 > 0$ such that $\mathcal{F}(x_0)=0$. Moreover, $x_0\in(0,x_1-x_2)$.
\item[(b)] Given $x_{1}\geq 0$, the equation $\mathcal{F}(x_0)=0$ admits a solution $x_0\in (x_{1},+\infty)$ if and only if $x_{2}<-2x_{1}$. Moreover, the solution $x_{0}$ is unique and satisfies $x_{0} \in (x_{1},x_{1}-x_{2})$.
\end{itemize}
\end{proposition}
\begin{proof}
(a) It is clear that $\mathcal{F}(x)$ is continuous and differentiable on its domain. We will show that $\mathcal{F}(0)<0$, $\mathcal{F}(x_1-x_2)>0$, and that $\mathcal{F}'(x)>0$ for all $x\in(0,\infty)$ and $\mathcal{F}'(x) \to + \infty$ as $x \to + \infty$, which implies the stated result. The inequality $\mathcal{F}(x_1-x_2)>0$ is easy to establish from a direct inspection of \eqref{Fdef}. Also, we note that
\begin{align}\label{F' lol}
\mathcal{F}'(x)=\int_{x_2}^{x_1}\frac{3x-x_1-x_2}{4\sqrt{(x-s)(x_1-s)(s-x_2)}}~ds,
\end{align}
which is clearly positive for $x>0$ because $x_2<x_1<0$, and $\mathcal{F}'(x) \sim \frac{3\pi}{4}\sqrt{x} \to + \infty$ as $x \to + \infty$. To show that $\mathcal{F}(0)<0$, let $x_*=\frac{1}{2}(x_1+x_2)$, and note that $\mathcal{F}(0)$ can be written as
\begin{align*}
\mathcal{F}(0)=\left(\int_{x_2}^{x_*}+\int_{x_*}^{x_1}\right)\frac{\sqrt{-s}\left(s-x_*\right)}{\sqrt{(x_1-s)(s-x_2)}}ds<0.
\end{align*}
Since $f(s)=\frac{s-x_*}{\sqrt{(x_1-s)(s-x_2)}}$ satisfies $f(x_{*}+s)=-f(x_{*}-s)$ for all $s \in (-\frac{x_{1}-x_{2}}{2},\frac{x_{1}-x_{2}}{2})$ and since $\sqrt{-s}$ is positive and decreases as $s\in(x_2,x_{1})$ increases, this implies that $\mathcal{F}(0)<0$. \\
(b) The numerator in \eqref{F' lol} is positive if and only if $x>\frac{x_{1}+x_{2}}{3}$. Since $x_{1} \geq 0$, we have $\frac{x_{1}+x_{2}}{3} < x_{1}$ and therefore $\mathcal{F}'(x) > 0$ for all $x \in (x_{1},+\infty)$. By a direct computation we get
\begin{align*}
\mathcal{F}(x_{1}) = \int_{x_{2}}^{x_{1}} \frac{s-\frac{x_{2}}{2}}{\sqrt{s-x_{2}}}ds = \frac{1}{3}\sqrt{x_{1}-x_{2}}(2x_{1}+x_{2}).
\end{align*}
This shows that $\mathcal{F}(x_{1})<0$, and therefore that there exists $x_{0} \in (x_{1},+\infty)$ such that $\mathcal{F}(x_{0})=0$, if and only if $x_{2}<-2x_{1}$. Since $x_{2}<0$, the inequality $\mathcal{F}(x_1-x_2)>0$ still holds, which implies that $x_{0} \in (x_{1},x_{1}-x_{2})$.
\end{proof}
Throughout the remainder of this paper, $x_0$ is defined as in Proposition \ref{prop: x0}. Since
\begin{equation*}
q(z)=\tilde{q}(z)(z-x_0), \qquad \tilde{q}(z):=-z-\frac{1}{2}(x_0-x_1-x_2),
\end{equation*}
and $x_2<-\frac{1}{2}(x_0-x_1-x_2)<x_1$, we have 
\begin{align}\label{qsigns}
q(x_{2})<0,  \qquad q(x_{1})>0, \qquad q(x_{0})=0, \qquad q'(x_{0})=\tilde{q}(x_{0})<0.
\end{align}

The square root $\sqrt{\mathcal{R}(z)} = \sqrt{z-x_0}\sqrt{z-x_1}\sqrt{z-x_2}$ is analytic on $\mathbb{C}\setminus \big( (-\infty,x_2] \cup [x_1,x_0] \big)$, behaves as $\sqrt{\mathcal{R}(z)} \sim z^\frac{3}{2}$ as $z \to \infty$, and satisfies the jump relation
\begin{align}\label{jump for R 2cuts}
& \sqrt{\mathcal{R}(z)}_{+} + \sqrt{\mathcal{R}(z)}_{-} = 0, & & z \in (-\infty,x_2)\cup(x_1,x_0).
\end{align}
We define the $\mathfrak{g}$-function by
\begin{align}\label{g function def 2 cut}
\mathfrak{g}(z) = \int_{x_0}^{z} \frac{q(s)}{\sqrt{\mathcal{R}(s)}}ds,
\end{align}
where $q$ is given by \eqref{qPolynomial}, and the path of integration does not cross $(-\infty,x_0]$. 
\begin{lemma}\label{g_lemma}
\begin{enumerate}
The $\mathfrak{g}$-function satisfies the following properties:
\item $\mathfrak{g}$ is analytic in $\mathbb{C}\setminus(-\infty,x_0]$ and satisfies $\mathfrak{g}(z) = \overline{\mathfrak{g}(\overline{z})}$ for $z\in\mathbb{C}\setminus(-\infty,x_0)$.
\item $\mathfrak{g}$ satisfies the jump conditions 
\begin{align}
& \mathfrak{g}_{+}(z) + \mathfrak{g}_{-}(z) = 0, & & z \in (-\infty,x_2)\cup(x_1,x_0), \label{jump1 g 2cuts} \\
& \mathfrak{g}_{+}(z) - \mathfrak{g}_{-}(z) = i \Omega, & & z \in (x_2,x_1), \label{jump3 g 2cuts}
\end{align}
where $\Omega = 2i\int_{x_1}^{x_0} \mathfrak{g}_{+}'(s)ds = -2 i \mathfrak{g}_{+}(x_1) > 0$.
\item As $z\to\infty$, we have
\begin{equation}\label{asymp g 2cuts}
\mathfrak{g}(z) = -\frac{2}{3}z^\frac{3}{2}+\frac{1}{4}\left((x_1-x_2)^2-x_0(3x_0-2x_1-2x_2)\right)z^{-\frac{1}{2}} + \bigO(z^{-\frac{3}{2}}).
\end{equation}
In particular, $\re \mathfrak{g}(z) \to +\infty$ as $z \to \infty$ along either of the two rays $\arg(z-x_2) = \pm 2\pi/3$.

\item There exists an open neighborhood $\mathcal{V} \subset \C$ of $(-\infty,x_0)$ and an $M>0$ such that 
\begin{align}\label{full sector at -inf}
\Big\{z \in \mathbb{C} : \arg(z-x_2) \in \Big(-\pi,-\frac{2\pi}{3}\Big]\cup \Big[\frac{2\pi}{3},\pi\Big) \mbox{ and } |z| \geq M \Big\} \subset \mathcal{V}
\end{align}
and such that
\begin{equation}\label{real g}
\re \mathfrak{g}(z) \geq0 \qquad \mbox{for all }z\in\mathcal{V},
\end{equation}
where equality holds in (\ref{real g}) if and only if $z\in(-\infty,x_2]\cup[x_1,x_0]$.
\end{enumerate}
\end{lemma}
\begin{proof}
The analyticity of $\mathfrak{g}$ follows from (\ref{g function def 2 cut}). The symmetry $\mathfrak{g}(z) = \overline{\mathfrak{g}(\overline{z})}$ follows from the fact that $q$ has real coefficients. 
The jumps \eqref{jump for R 2cuts} combined with $\int_{x_2}^{x_1} \frac{q(s)}{\sqrt{\mathcal{R}(s)}} ds=-\mathcal{F}(x_0)=0$ imply \eqref{jump1 g 2cuts}. The jump \eqref{jump3 g 2cuts} is a consequence of \eqref{jump for R 2cuts} and \eqref{g function def 2 cut}; the fact that $\Omega > 0$ follows from (\ref{qsigns}) and (\ref{g function def 2 cut}). Expanding $\mathfrak{g}'(z)$ as $z \to \infty$, and then integrating these asymptotics gives
\begin{equation*}
\mathfrak{g}(z) = -\frac{2}{3}z^\frac{3}{2} + \mathfrak{g}_{0} +\frac{1}{4}\left((x_1-x_2)^2-x_0(3x_0-2x_1-2x_2)\right)z^{-\frac{1}{2}} + \bigO(z^{-\frac{3}{2}}) \qquad \mbox{as } z \to \infty,
\end{equation*}
for some $\mathfrak{g}_{0} \in \mathbb{C}$. Equation \eqref{jump1 g 2cuts} implies that $\mathfrak{g}_{0}=0$, which proves (\ref{asymp g 2cuts}). Equation (\ref{asymp g 2cuts}) implies that there exists an $M>0$ such that $\re \mathfrak{g}(z)\geq 0$ for all $z \in \C$ with $|z|\geq M$ and $\arg z \in (-\pi,-\frac{2\pi}{3}] \cup [\frac{2\pi}{3},\pi)$. Also, from (\ref{jump1 g 2cuts}) and the symmetry $\mathfrak{g}(z) = \overline{\mathfrak{g}(\overline{z})}$, we deduce that $\re \mathfrak{g}_+(z) = \re \mathfrak{g}_-(z)=0$ for $z \in (-\infty,x_{2}]\cup [x_{1},x_{0}]$. 
Since $q(z)<0$ for $z \in (-\infty,x_{2})$ and $q(z)>0$ for $z \in (x_{1},x_{0})$ by (\ref{qsigns}), the imaginary part of $\mathfrak{g}_{+}(x)$ is decreasing as $x \in (-\infty,x_{2})\cup(x_{1},x_{0})$ increases. Hence, by the Cauchy-Riemann equations, for each $x \in (-\infty,x_{2})\cup(x_{1},x_{0})$, there exists $\epsilon = \epsilon(x)>0$ such that $\re \mathfrak{g}(x + i u \epsilon(x))>0$ for all $u \in (0,1]$. Assertion 4 will follow if we can show that $\epsilon$ can be chosen independently of $x$. This can be achieved by a local analysis of $\mathfrak{g}$ near each of the points $x_{0}, x_{1}, x_{2}$. As $z \to x_{0}$, since $q(x_{0})=0$, we have $\mathfrak{g}(z) \sim \frac{2}{3}q'(x_{0})(z-x_{0})^{3/2}/\sqrt{(x_{0}-x_{1})(x_{0}-x_{2})} $. Since $q'(x_{0})<0$, there exists a small neighborhood $V_{0}$ of $x_{0}$ such that $\re \mathfrak{g}(z) \geq 0$ for $z \in V_{0} \cap \{\re z \leq x_{0}\}$, with equality only if $z \in [x_1,x_{0}]\cap V_0$. As $z \to x_{1}$, we have $\mathfrak{g}(z) \sim 2 q(x_{1})\sqrt{z-x_{1}}/\sqrt{(x_{1}-x_{2})(x_{0}-x_{1})}$. Recalling that $q(x_{1})>0$, $\re \mathfrak{g}(z) \geq 0$ in a full open neighborhood $V_1$ of $x_{1}$, with equality only if $z \in [x_{1},x_0]\cap V_1$. The local analysis near $x_{2}$ is similar.
\end{proof}

\begin{remark}\label{OmegaRemark}
For future reference, we note that the equation $F(x_{0})=0$ can be rewritten as
\begin{equation}\label{x0 implicit eq}
\frac{\textbf{E}(k)}{\textbf{K}(k)}=\frac{-2(x_0-x_1)}{x_0+x_1+x_2}, \qquad \mbox{where} \qquad k:=\sqrt{\frac{x_1-x_2}{x_0-x_2}},
\end{equation}
and $\textbf{K}$, $\textbf{E}$ are the complete elliptic integrals of the first and second kind (see \cite[Eqs. 8.111.2 and 8.111.3]{GRtable})
\begin{align*}
\textbf{K}(k):=\int_0^{\frac{\pi}{2}}\frac{dt}{\sqrt{1-k^2\sin^2t}}, ~~~ \textbf{E}(k):=\int_0^{\frac{\pi}{2}} \sqrt{1-k^2\sin^2t}~dt.
\end{align*}
Also, using \cite[Eqs. 3.131.6, 3.131.11, 3.132.5 and 3.141.23]{GRtable} and \eqref{x0 implicit eq}, the constant $\Omega$ can be written as 
\begin{align}\label{OmegaEllipticInt}
    \Omega&=\frac{2}{3}\sqrt{x_0-x_2}(x_0+x_1+x_2)\left[\textbf{K}(k')\left(1-\frac{\textbf{E}(k)}{\textbf{K}(k)}\right)-\textbf{E}(k')\right], 
\end{align}
where $k':=\sqrt{\frac{x_0-x_1}{x_0-x_2}}=\sqrt{1-k^2}$.
\end{remark}

\subsection{Rescaling and opening of the lenses}\label{Subsection: T to S}
The first transformation $\Psi \to T$ of the steepest descent analysis is defined by
\begin{equation}\label{def of T}
T(z) = \begin{pmatrix}
r^{-\frac{1}{4}} & \frac{i}{4}\left(2x_0(x_1+x_2)-3x_0^2+(x_1-x_2)^2\right)r^{\frac{7}{4}} \\
0 & r^{\frac{1}{4}}
\end{pmatrix} \Psi(rz;rx_{1},rx_{2})e^{-r^\frac{3}{2}\mathfrak{g}(z)\sigma_{3}}.
\end{equation}
The asymptotics of $T(z)$ as $z \to \infty$ can be computed using the asymptotics \eqref{eq:psiasympinf} of $\Psi$ and \eqref{asymp g 2cuts} of $\mathfrak{g}$. The first matrix on the right-hand side of \eqref{def of T} is chosen to compensate for the behavior \eqref{asymp g 2cuts} of $\mathfrak{g}(z)$ as $z\to\infty$. It follows that
\begin{equation}
\label{eq:Tasympinf}
T(z) = \left( I + \Or\left(z^{-1}\right) \right) z^{\frac{\sigma_3}{4}} M^{-1} \qquad \mbox{as } z \to \infty,
\end{equation}
where the complex powers are defined using the principal branch. Using \eqref{jump1 g 2cuts}, we note the following factorization of the jump $T_{-}(z)^{-1}T_{+}(z)$ for $z \in (x_1,x_0)$:
\begin{align}\label{factorization jumps}
\begin{pmatrix}
e^{-2 r^\frac{3}{2}\mathfrak{g}_{+}(z)} & 1 \\
0 & e^{-2 r^\frac{3}{2}\mathfrak{g}_{-}(z)}
\end{pmatrix} = \begin{pmatrix}
1 & 0 \\
e^{-2 r^\frac{3}{2}\mathfrak{g}_{-}(z)} & 1
\end{pmatrix} \begin{pmatrix}
0 & 1 \\
-1 & 0
\end{pmatrix} \begin{pmatrix}
1 & 0 \\
e^{-2r^\frac{3}{2}\mathfrak{g}_{+}(z)} & 1
\end{pmatrix}.
\end{align}
Let $\tilde\gamma_{+}$ and $ \tilde\gamma_{-}$ be two simple curves oriented from $x_{1}$ to $x_{0}$ lying in the upper and lower half-planes, respectively, see Figure \ref{fig:contour for S}. 
\begin{figure}
\centering
\begin{tikzpicture}
\draw[fill] (0,0) circle (0.05);
\draw (0,0) -- (5,0);
\draw (0,0) -- (120:3);
\draw (0,0) -- (-120:3);
\draw (0,0) -- (-3,0);
\draw (5,0) -- (8,0);

\draw (3,0) .. controls (3.5,1) and (4.5,1) .. (5,0);
\draw (3,0) .. controls (3.5,-1) and (4.5,-1) .. (5,0);

\draw[fill] (3,0) circle (0.05);
\draw[fill] (5,0) circle (0.05);

\node at (0.15,-0.3) {$x_2$};
\node at (2.9,-0.3) {$x_1$};
\node at (5.2,-0.3) {$x_0$};
\node at (-3,-0.3) {$-\infty$};
\node at (8,-0.3) {$+\infty$};

\draw[black,arrows={-Triangle[length=0.18cm,width=0.12cm]}]
(-120:1.5) --  ++(60:0.001);
\draw[black,arrows={-Triangle[length=0.18cm,width=0.12cm]}]
(120:1.3) --  ++(-60:0.001);
\draw[black,arrows={-Triangle[length=0.18cm,width=0.12cm]}]
(180:1.5) --  ++(0:0.001);

\draw[black,arrows={-Triangle[length=0.18cm,width=0.12cm]}]
(0:4.08) --  ++(0:0.001);
\draw[black,arrows={-Triangle[length=0.18cm,width=0.12cm]}]
(0:1.6) --  ++(0:0.001);

\draw[black,arrows={-Triangle[length=0.18cm,width=0.12cm]}]
(4.08,0.75) --  ++(0:0.001);
\draw[black,arrows={-Triangle[length=0.18cm,width=0.12cm]}]
(4.08,-0.75) --  ++(0:0.001);
\draw[black,arrows={-Triangle[length=0.18cm,width=0.12cm]}]
(6.5,0) --  ++(0:0.001);

\node at (4,1) {$\tilde\gamma_{+}$};
\node at (4,-1) {$\tilde\gamma_{-}$};
\end{tikzpicture}
\caption{The jump contour $\Sigma_{S}$.}
\label{fig:contour for S}
\end{figure}
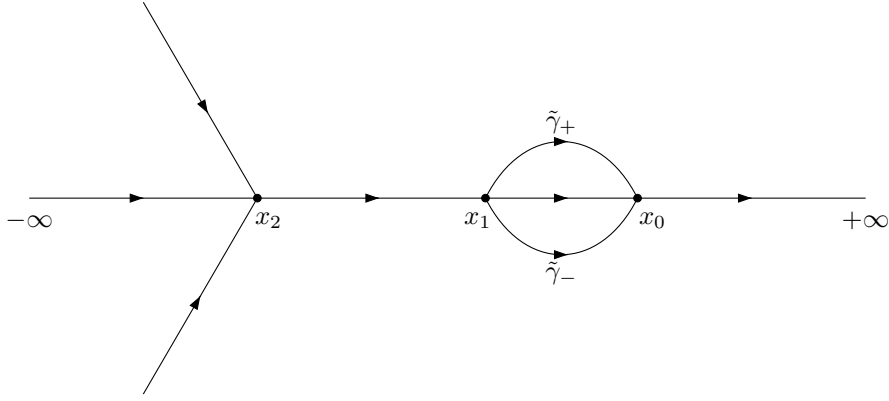
The second transformation $T \mapsto S$ is defined by
\begin{equation}\label{def:S}
S(z)=T(z)\left\{ \hspace{-0.1cm} \begin{array}{l l}
\begin{pmatrix}
1 & 0 \\
-e^{-2r^\frac{3}{2}\mathfrak{g}(z)} & 1
\end{pmatrix}, & z \mbox{ is below } \tilde{\gamma}_{+} \mbox{ and } \im z  > 0, \\
\begin{pmatrix}
1 & 0 \\
e^{-2r^\frac{3}{2}\mathfrak{g}(z)} & 1
\end{pmatrix}, & z \mbox{ is above } \tilde{\gamma}_{-} \mbox{ and } \im z < 0, \\
I, & \mbox{otherwise}.
\end{array} \right.
\end{equation}
$S$ satisfies the following RH problem, whose properties can be deduced from those of $\Psi$, combined with Lemma \ref{g_lemma}, the definition \eqref{def of T} of $T$ and the factorization \eqref{factorization jumps}. 
\subsubsection*{RH problem for $S$}
\begin{enumerate}[label={(\alph*)}]
\item[(a)] $S : \C \backslash \Sigma_{S} \rightarrow \C^{2\times 2}$ is analytic, where $\Sigma_{S}:=\mathbb{R}\cup\gamma_{+}\cup \gamma_{-}$ and $\gamma_{\pm} := \tilde\gamma_{\pm} \cup \big(e^{\pm \frac{2\pi i}{3}} (-\infty,x_2)\big)$. The orientation of $\Sigma_{S}$ is shown in Figure \ref{fig:contour for S}.
\item[(b)] The jumps for $S$ are given by
\begin{align}
& S_{+}(z) = S_{-}(z)\begin{pmatrix}
0 & 1 \\ -1 & 0
\end{pmatrix}, & & z \in (-\infty,x_2)\cup (x_1,x_0), \label{SInf0 jump} \\
& S_{+}(z) = S_{-}(z)\begin{pmatrix}
1 & 0 \\
e^{-2r^\frac{3}{2}\mathfrak{g}(z)} & 1
\end{pmatrix}, & & z \in \gamma_{+}\cup \gamma_{-}, \label{Slense jump} \\
& S_{+}(z) = S_{-}(z)e^{-i \Omega r^\frac{3}{2}\sigma_{3}}, & & z \in (x_2,x_1), \label{S0y2 jump} \\
& S_{+}(z) = S_{-}(z)\begin{pmatrix}
1 & e^{2r^{\frac{3}{2}}\mathfrak{g}(z)} \\
0 & 1
\end{pmatrix}, & & z \in (x_0,\infty).
\end{align}
\item[(c)] As $z \rightarrow \infty$, we have $S(z) = \left( I + \Or\left(z^{-1}\right) \right) z^{\frac{\sigma_3}{4}} M^{-1}.$
\item[(d)] As $z \to x_j$, $j=0,1,2$, we have $S(z) = \Or( \log(z-x_j))$.
\end{enumerate}
Choose $\mathcal{V}$ as in Lemma \ref{g_lemma}. Deforming the contours $\gamma_{+}$ and $\gamma_{-}$ if necessary, we may assume that they lie in $\mathcal{V}$. Since $\re \mathfrak{g}(z) \geq 0$ for all $z \in \mathcal{V}$ with equality only if $z \in (-\infty,x_{2}]\cup[x_{1},x_{0}]$,
the jumps for $S$ are exponentially close to $I$ as $r \to + \infty$ on $\gamma_{+}\cup \gamma_{-}$. This convergence is uniform except for $z$ in small neighborhoods of $x_j$, $j=0,1,2$.

\section{Global parametrix}\label{section: global parametrix}
Ignoring the exponentially small jumps for $S$, and ignoring small neighborhoods of $x_{0},x_{1},x_{2}$, we are led to consider the following RH problem.

\subsubsection*{RH problem for $P^{(\infty)}$}
\begin{enumerate}[label={(\alph*)}]
\item[(a)] $P^{(\infty)} : \C \backslash (-\infty,x_0] \rightarrow \C^{2\times 2}$ is analytic.
\item[(b)] The jumps for $P^{(\infty)}$ are given by
\begin{align}
& P^{(\infty)}_{+}(z) = P^{(\infty)}_{-}(z)\begin{pmatrix}
0 & 1 \\ -1 & 0
\end{pmatrix}, & & z \in (-\infty,x_2)\cup (x_1,x_0), \nonumber \\
& P^{(\infty)}_{+}(z) = P^{(\infty)}_{-}(z)e^{-i\Omega r^\frac{3}{2} \sigma_{3}}, & & z \in (x_2,x_1). \label{jumps for Pinf on (x2,x1)}
\end{align}
\item[(c)] As $z \rightarrow \infty$, we have $P^{(\infty)}(z) = \left( I + \Or\left(z^{-1}\right) \right) z^{\frac{\sigma_3}{4}} M^{-1}$.
\item[(d)] As $z \to x_j$, $j=0,1,2$, we have $P^{(\infty)}(z) = \bigO((z-x_{j})^{-\frac{1}{4}})$.
\end{enumerate}
The above RH problem was solved in \cite{BCL19} (the points $y_{1}$, $y_{2}$ and $0$ in \cite{BCL19} should be identified with $x_{0}$, $x_{1}$ and $x_{2}$ in the present paper, respectively). Consider the function $\varphi : \mathbb{C}\setminus (-\infty,x_{0}] \to \mathbb{C}$ defined by 
\begin{align}\label{def of u}
\varphi(z) = \int_{x_0}^{z} \omega, \qquad \omega = \frac{c_{0}dz}{\sqrt{\mathcal{R}(z)}}, \qquad c_{0} =\frac{\sqrt{x_0-x_2}}{4\textbf{K}(k)},
\end{align}
where the path of integration lies in $\mathbb{C}\setminus (-\infty,x_{0}]$, and $k$ has been defined in \eqref{x0 implicit eq}. Note that the holomorphic differential $\omega_{1}$ of \eqref{def of omegaj} with $g=1$ coincides with $\omega$. By \cite{BCL19}, $P^{(\infty)}$ is given by
\begin{align}
&P^{(\infty)}(z)=\begin{pmatrix}
\ds \frac{1}{\mathcal{G}(\frac{1}{2})} & \ds \frac{-ic_{\mathcal{G}}}{\mathcal{G}(0)} \\
0 & \ds \frac{1}{\mathcal{G}(0)}
\end{pmatrix}\frac{\beta(z)^{\sigma_3}}{\sqrt{2}}\begin{pmatrix}
\ds \mathcal{G}(-\varphi(z)) & \ds -i\mathcal{G}(\varphi(z)) \\
\ds -i\mathcal{G}(-\varphi(z)-\tfrac{1}{2}) & \ds \mathcal{G}(\varphi(z)-\tfrac{1}{2}) 
\end{pmatrix} \label{def of P inf hat two cuts FINAL}
\end{align}
with
\begin{align}\label{nu_r_relation}
& \mathcal{G}(z)=\frac{\theta(z + \nu)}{\theta(z)}, \quad \nu = - \frac{\Omega}{2\pi}r^\frac{3}{2}, \quad c_{\mathcal{G}} = 2c_{0} (\log\mathcal{G})'(\tfrac{1}{2}), \quad \beta(z) = \frac{(z-x_0)^{1/4}(z-x_2)^{1/4}}{(z-x_1)^{1/4}},
\end{align}
where the principal branch is taken for the roots. The function $\theta$ is defined by \eqref{def of theta genus g} with $g=1$, and is associated to the parameter $\tau \in i \mathbb{R}^{+}$ defined in \eqref{def of tau genus g}.


\section{Local parametrices}\label{subsec:Localparametrix}
The goal of this section is to construct local approximations (called ``parametrices") of $S$ in small open disks $\mathbb{D}_{x_j}$ centered at $x_{j}$, $j=0,1,2$. The local parametrix $P^{(x_j)}$ possesses the same jumps as $S$ inside $\mathbb{D}_{x_j}$, $P^{(x_j)}=\bigO(\log(z-x_j))$ as $z\to x_j$, and satisfies $P^{(x_{j})}(z)P^{(\infty)}(z)^{-1} = I+o(1)$ as $r\to +\infty$ uniformly for $z \in \partial \mathbb{D}_{x_{j}}$. In our case, the local parametrices are standard: $P^{(x_0)}$ can be built in terms of Airy functions (as in \cite{DKMVZ1}), and $P^{(x_1)}$ and $P^{(x_2)}$ in terms of Bessel functions (as in \cite{DIZ}). 

\subsection{Parametrix at $x_0$}

The function $f_{0}(z):=\left(-\frac{3}{2}\mathfrak{g}(z)\right)^\frac{2}{3}$ is a conformal map from $\mathbb{D}_{x_{0}}$ to a neighborhood of $0$, and satisfies
\begin{align}
    f_{0}(z)&=c_{x_{0}}(z-x_{0})(1+c_{x_0}^{(2)}(z-x_0)+c_{x_0}^{(3)}(z-x_0)^2+\bigO((z-x_{0})^3)) \qquad \mbox{as } z \to x_{0}, \label{x0coordExpansion} \\
    c_{x_{0}}&=\frac{(-\tilde{q}(x_0))^\frac{2}{3}}{(x_0-x_1)^\frac{1}{3}(x_0-x_2)^\frac{1}{3}}>0, \qquad c_{x_0}^{(2)}=\frac{-2x_0+x_1+x_2}{5(x_0-x_1)(x_0-x_2)}-\frac{2}{5\tilde{q}(x_0)}, \\
    175c_{x_0}^{(3)}&=\frac{43x_0^2-43x_0x_1+17x_1^2-43x_0x_2+9x_1x_2+17x_2^2}{(x_0-x_1)^2(x_0-x_2)^2}+\frac{18(2x_0-x_1-x_2)}{(x_0-x_1)(x_0-x_2)\tilde{q}(x_0)}-\frac{7}{\tilde{q}(x_0)^2}, \nonumber
\end{align}
where we recall that $-\tilde q(x_0)>0$. The model RH problem of \cite{DKMVZ1}, which is needed for the construction of $P^{(x_{0})}$, is presented in Appendix \ref{subsec:Airyparametrix} for convenience; its solution is denoted by $\Phi_{\mathrm{Ai}}$. Deforming $\tilde\gamma_{\pm}$ if necessary, we may assume that $f_0(\tilde\gamma_{\pm}\cap\mathbb{D}_{x_{0}})\subset e^{\pm\frac{2\pi i}{3}}\mathbb{R}^+$. It can be verified that
\begin{align}
& P^{(x_{0})}(z) = E_{x_{0}}(z)\Phi_{\mathrm{Ai}}(rf_{0}(z))e^{-r^{\frac{3}{2}}\mathfrak{g}(z)\sigma_{3}}, & & E_{x_{0}}(z) = P^{(\infty)}(z)M^{-1}\Big(rf_{0}(z)\Big)^{\frac{\sigma_{3}}{4}}, \label{Px0}
\end{align}
where $E_{x_0}(z)$ is analytic for $z\in\mathbb{D}_{x_0}$. Furthermore, due to \eqref{x0coordExpansion} and \eqref{Asymptotics Airy}, 
\begin{equation}\label{Px0asymp}
P^{(x_{0})}(z)P^{(\infty)}(z)^{-1}=I+\frac{P^{(\infty)}(z)\Phi_{\mathrm{Ai},1}P^{(\infty)}(z)^{-1} }{r^\frac{3}{2}f_{0}(z)^\frac{3}{2}} + \bigO(r^{-3})
\end{equation}
as $r\to +\infty$ uniformly for $z\in\partial\mathbb{D}_{x_{0}}$.

\subsection{Parametrices at $x_1$ and $x_2$}
We define $\tilde{f}(z) = \frac{1}{4}\left(\mathfrak{g}(z)\mp\frac{i\Omega}{2}\right)^2$, where we take the $-/+$ sign when $z$ is above/below the real axis. As $z\to x_j$, $j=1,2,$ we have
\begin{align}
& \tilde{f}(z) =-c_{x_1}(z-x_1)\big(1+c_{x_{1}}^{(2)}(z-x_1)+\bigO((z-x_1)^{2})\big), & & c_{x_1}=\frac{q^2(x_1)}{(x_1-x_2)(x_0-x_1)}, & & \mbox{as } z \to x_{1}, \label{cx1} \\
& \tilde{f}(z) =c_{x_{2}}(z-x_{2})\big(1+c_{x_{2}}^{(2)}(z-x_{2})+\bigO((z-x_{2})^{2})\big), & & c_{x_{2}}=\frac{q^2(x_{2})}{(x_0-x_2)(x_1-x_2)}, & & \mbox{as } z \to x_{2}. \label{cx2}
\end{align}
Again, deforming the lenses if necessary, we may assume that $\tilde{f}(\tilde\gamma_{\pm}\cap\mathbb{D}_{x_{1}})\subset e^{\mp\frac{2\pi i}{3}}\mathbb{R}^+$ and $\tilde{f}(\gamma_{\pm}\cap\mathbb{D}_{x_{2}})\subset e^{\pm\frac{2\pi i}{3}}\mathbb{R}^+$. The local parametrices $P^{(x_1)}$ and $P^{(x_2)}$ are given by
\begin{align}
& P^{(x_1)}(z)=E_{x_1}(z)\sigma_{3}\Phi_{\mathrm{Be}}(r^{3} \tilde{f}(z))\sigma_{3}e^{-r^{\frac{3}{2}}g(z)\sigma_{3}}, & & E_{x_1}(z)=P^{(\infty)}(z)e^{\pm\frac{i \Omega}{2}r^{\frac{3}{2}}\sigma_{3}}M\Big(2\pi r^{\frac{3}{2}}\tilde{f}(z)^{\frac{1}{2}} \Big)^{\frac{\sigma_{3}}{2}}, \label{Px1} \\
&  P^{(x_{2})}(z)=E_{x_{2}}(z)\Phi_{\mathrm{Be}}(r^{3} \tilde{f}(z))e^{-r^{\frac{3}{2}}g(z)\sigma_{3}}, & & E_{x_{2}}(z) = P^{(\infty)}(z)e^{\pm \frac{i \Omega}{2}r^{\frac{3}{2}}\sigma_{3}}M^{-1}\Big(2\pi r^{\frac{3}{2}}\tilde{f}(z)^{\frac{1}{2}} \Big)^{\frac{\sigma_{3}}{2}}, \label{Px2}
\end{align}
where $E_{x_j}(z)$ is analytic in $\mathbb{D}_{x_j}$, $j=1,2$, and $\Phi_{\mathrm{Be}}$ is the solution of the Bessel model RH problem recalled in Appendix \ref{subsec:Besselparametrix}. As $r \to + \infty$, we have $P^{(x_{j})}(z)P^{(\infty)}(z)^{-1} = I + J_{R}^{(1)}(z)r^{-\frac{3}{2}} + \bigO(r^{-3})$ uniformly for $z \in \partial\mathbb{D}_{x_{j}}$, for a certain matrix $J_{R}^{(1)}(z)$ that will be computed in \eqref{explicit expression for the jumps Jrp1p}.

\section{Small norm problem}\label{section:smallnorm}
Let us define
\begin{equation}\label{errorMatrix}
R(z)=\begin{cases}
        S(z)P^{(x_j)}(z)^{-1}, &z\in\mathbb{D}_{x_j}, ~ j=0,1,2, \\
        S(z)P^{(\infty)}(z)^{-1}, &z\in\mathbb{C}\setminus\bigcup_{j=0}^2 \overline{\mathbb{D}}_{x_j}.
    \end{cases}
\end{equation}
From the asymptotics of $S(z)$ and $P^{(\infty)}(z)$ as $z \to \infty$, we have $R(z) = I + \bigO(z^{-1})$ as $z \to \infty$. Also, by definition of $P^{(x_j)}$, $j=0,1,2$, $R(z)$ is analytic for $z\in \cup_{j=0}^{2}\mathbb{D}_{x_j}$. Therefore, the jump contour for $R$, denoted by $\Sigma_{R}$, is given by
\begin{align*}
\Sigma_{R} = \bigg((x_{0},+\infty) \cup \gamma_{+} \cup \gamma_{-}  \cup \bigcup_{j=0}^{2} \partial \mathbb{D}_{x_{j}} \bigg)  \setminus \bigcup_{j=0}^{2}  \mathbb{D}_{x_j},
\end{align*}
where we orient (for convenience) the boundaries of the disks in the clockwise direction, see Figure \ref{fig:SigmaR}. 
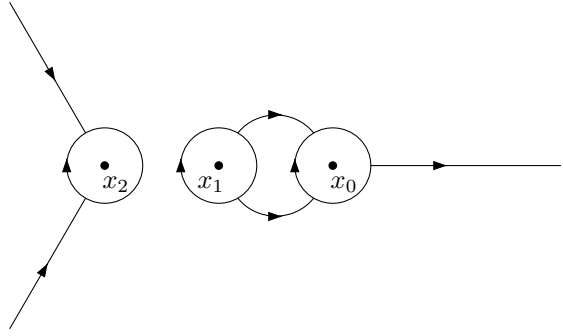
\begin{figure}
\centering
\begin{tikzpicture}

\draw[fill] (0,0) circle (0.05);
\draw[fill] (1.5,0) circle (0.05);
\draw[fill] (3,0) circle (0.05);
\draw (-0.25,0.43) -- (120:2.5);
\draw (-0.25,-0.43) -- (-120:2.5);
\draw (3.5,0) -- (6,0);

\draw (0,0) circle (0.5);
\draw (1.5,0) circle (0.5);
\draw (3,0) circle (0.5);

\draw (1.75,0.43) .. controls (2,0.75) and (2.5,0.75) .. (2.75,0.43);
\draw (1.75,-0.43) .. controls (2,-0.75) and (2.5,-0.75) .. (2.75,-0.43);

\node at (0.15,-0.25) {$x_2$};
\node at (1.4,-0.25) {$x_1$};
\node at (3.15,-0.25) {$x_0$};

\draw[black,arrows={-Triangle[length=0.18cm,width=0.12cm]}]
(-120:1.5) --  ++(60:0.001);
\draw[black,arrows={-Triangle[length=0.18cm,width=0.12cm]}]
(120:1.3) --  ++(-60:0.001);
\draw[black,arrows={-Triangle[length=0.18cm,width=0.12cm]}]
(0:4.5) --  ++(0:0.001);
\draw[black,arrows={-Triangle[length=0.18cm,width=0.12cm]}]
(2.33,0.675) --  ++(0:0.001);
\draw[black,arrows={-Triangle[length=0.18cm,width=0.12cm]}]
(2.33,-0.675) --  ++(0:0.001);
\draw[black,arrows={-Triangle[length=0.18cm,width=0.12cm]}]
(-0.5,0.09) --  ++(90:0.001);
\draw[black,arrows={-Triangle[length=0.18cm,width=0.12cm]}]
(1,0.09) --  ++(90:0.001);
\draw[black,arrows={-Triangle[length=0.18cm,width=0.12cm]}]
(2.5,0.09) --  ++(90:0.001);

\end{tikzpicture}
\caption{The jump contour $\Sigma_{R}$.}
\label{fig:SigmaR}
\end{figure}
It follows from the steepest descent analysis that
\begin{equation}\label{JR expansion}
J_{R}(z) := R_{-}(z)^{-1}R_{+}(z) = \begin{cases}
I + \bigO(e^{-\tilde{c} |rz|^{\frac{3}{2}}}), & \mbox{uniformly for } z \in \Sigma_{R}  \cup_{j=0}^{2} \partial \mathbb{D}_{x_{j}}, \\
I + \frac{J_{R}^{(1)}(z)}{r^{\frac{3}{2}}} + \bigO(r^{-3}), & \mbox{uniformly for } z \in \cup_{j=0}^{2} \partial \mathbb{D}_{x_{j}},
\end{cases}
\end{equation}
where $\tilde{c}>0$ is a sufficiently small constant, and $J_{R}^{(1)}(z)$ is given by
\begin{align}
J_{R}^{(1)}(z) = \begin{cases}
\frac{P^{(\infty)}(z)\Phi_{\mathrm{Ai,1}} P^{(\infty)}(z)^{-1}}{f_{0}(z)^{\frac{3}{2}}} & \mbox{if } z  \in \partial \mathbb{D}_{x_{0}}, \\[0.2cm]
\frac{P^{(\infty)}(z)e^{\pm \frac{i \Omega}{2}r^{\frac{3}{2}}\sigma_{3}}\sigma_{3}^{j} \Phi_{\mathrm{Be},1}\sigma_{3}^{j} e^{\mp \frac{i\Omega}{2}r^{\frac{3}{2}}}P^{(\infty)}(z)^{-1} }{\tilde{f}(z)^{1/2}} & \mbox{if } z  \in \partial \mathbb{D}_{x_{j}}, \; j=1,2.
\end{cases} \label{explicit expression for the jumps Jrp1p}
\end{align}
Since $\Omega \in \mathbb{R}$, $J_{R}^{(1)}(z) = \bigO(1)$ as $r \to + \infty$ uniformly for $z \in \cup_{j=0}^{2} \partial \mathbb{D}_{x_{j}}$. The jump matrix $J_{R}$ is uniformly close to the identity, thus $R(z)$ exists for sufficiently large $r$ and \cite{DKMVZ1,DKMVZ2,DeiftZhou}
\begin{equation}\label{large r asymptotics for R}
R(z)=I+\frac{R^{(1)}(z)}{r^\frac{3}{2}}+\bigO(r^{-3}), \qquad \mbox{where } \quad R^{(1)}(z)=\frac{1}{2\pi i}\int_{\bigcup_{j=0}^2\partial\mathbb{D}_{x_j}}\frac{J_R^{(1)}(\xi)}{\xi-z}~d\xi,
\end{equation}
as $r \to + \infty$ uniformly for $z\in\mathbb{C}\setminus\Sigma_{R}$, and these asymptotics can be differentiated with respect to $z$ without changing the error term. In Section \ref{section:integration1}, we will need explicit expressions for
\begin{align}\label{sandwich j}
\Big[E_{x_j}(x_j)^{-1} R^{(1)\prime}(x_j)E_{x_j}(x_j)\Big]_{21}, \qquad j=1,2.
\end{align} 
These quantities can be computed from \eqref{explicit expression for the jumps Jrp1p} and \eqref{large r asymptotics for R} by residue calculations. Let us define $\mathcal{J}_{x_{j}}(z) := [E_{x_j}(x_j)^{-1} J_{R}^{(1)}(z)E_{x_j}(x_j)]_{21}$ for $j=1,2$. By \eqref{explicit expression for the jumps Jrp1p}, as $z \to x_{j'}$, $j'=0,1,2$, we have
\begin{align}
& \mathcal{J}_{x_{j}}(z) = \sum_{k=-2}^{1} (\mathcal{J}_{x_{j}})_{x_{j'}}^{(k)}(z-x_{j'})^{k} + \bigO((z-x_{j'})^{2}), & & (\mathcal{J}_{x_{j}})_{x_{1}}^{(-2)} = (\mathcal{J}_{x_{j}})_{x_{2}}^{(-2)} = 0, \label{Jcal coeff def}
\end{align}
for certain matrices $(\mathcal{J}_{x_{j}})_{x_{j'}}^{(k)}$ that can be computed if needed. Since the circles $\partial\mathbb{D}_{x_j}$, $j=0,1,2,$ have clockwise orientation in \eqref{large r asymptotics for R}, we obtain
\begin{align}\label{sandwich in terms of Jcal}
\Big[E_{x_j}(x_j)^{-1} R^{(1)\prime}(x_j)E_{x_j}(x_j)\Big]_{21} = -(\mathcal{J}_{x_{j}})_{x_{j}}^{(1)}+\frac{2(\mathcal{J}_{x_{j}})_{x_0}^{(-2)}}{(x_0-x_j)^3}-\sum_{\substack{j'=0 \\ j'\neq j}}^2\frac{(\mathcal{J}_{x_{j}})_{x_{j'}}^{(-1)}}{(x_{j}-x_{j'})^{2}}, \qquad j=1,2.
\end{align}

\section{Proof of Theorem \ref{thm: main result}}\label{section:integration1}
Substituting the transformations $\Psi \mapsto T \mapsto S \mapsto R$ of Sections \ref{section:RH1}-\ref{section:smallnorm} into the differential identity (\ref{final formula diff identity}) and using the expansion \eqref{large r asymptotics for R} of $R$, we obtain the following asymptotics:
\begin{align}\label{partialrlogdet}
\partial_r\log F(r (x_{2},x_{1})) 
    = &\; I_1(r)  +  I_2(r)  +  I_3(r) + \bigO \big(r^{-\frac{5}{2}}\big) \quad \text{as $r \to +\infty$},
\end{align}
where 
\begin{align*}
& I_1(r) := r^2 \sum_{j=1}^2 (-1)^{j} x_j c_{x_j},
\qquad I_2(r) := \frac{1}{2\pi i r^{\frac{5}{2}}}\sum_{j=1}^2 (-1)^{j}x_j\Big[E_{x_j}(x_j)^{-1} R^{(1)\prime}(x_j)E_{x_j}(x_j)\Big]_{21},
	\\
& I_3(r) := \frac{1}{2\pi i r}  \sum_{j=1}^2 (-1)^{j}x_j \Big[ E_{x_j}(x_j)^{-1} E_{x_j}'(x_j)\Big]_{21},
\end{align*}
with $c_{x_j}$, $j=1,2$, given by \eqref{cx1}-\eqref{cx2}.
A straightforward calculation shows that
\begin{align}\label{I1final}
I_1(r) = c\partial_{r} r^{3}, 
\end{align}
where
\begin{align}\label{cConst}
    c =\frac{1}{12}\left[x_0^3+x_1^3+x_2^3-(x_0+x_1)(x_0+x_2)(x_1+x_2)\right]-\frac{q_0}{3}(x_0+x_1+x_2),
\end{align}
and the constant $q_0$ is defined in \eqref{qPolynomial}. By a simple rearrangement of the terms, it is easy to see that $c$ in \eqref{cConst} coincides with the constant $c$ defined in \eqref{def of C1 genus g} for $g=1$. This establishes the leading term in (\ref{F asymp genus g}) for $g = 1$; the two subleading terms are more complicated to evaluate.

\subsection{Explicit expression for $I_3(r)$}\label{subsection: third term}
From the definition of $E_{x_{j}}$, $j=1,2$, it is possible to obtain explicit expressions for $E_{x_j}(x_j)$ and $E_{x_j}'(x_j)$. These expressions are rather long, so we only write down $E_{x_1}(x_1)$ and $E_{x_1}'(x_1)$ (the expressions for $E_{x_2}(x_2)$ and $E_{x_2}'(x_2)$ are similar):  
\begin{multline*}
E_{x_1}(x_1)=e^{i\pi\nu}e^{-\frac{i\pi}{4}\sigma_3}\Bigg[\frac{\beta_{x_1}^{(-\frac{1}{4})}}{\mathcal{G}(\frac{1}{2})}\begin{pmatrix} \mathcal{G}(\frac{\tau}{2}) & \varphi_{x_1}^{(\frac{1}{2})}\mathcal{G}'(\frac{\tau}{2}) \\ 0 & 0 \end{pmatrix} \\
     +\frac{1}{\beta_{x_1}^{(-\frac{1}{4})}\mathcal{G}(0)}\begin{pmatrix} c_\mathcal{G} & c_\mathcal{G} \\ 1 & 1 \end{pmatrix} \begin{pmatrix} \frac{\tilde{\mathcal{G}}(\frac{1+\tau}{2})}{\varphi_{x_1}^{(\frac{1}{2})}} & 0 \\ 0 & \tilde{\mathcal{G}}'(\frac{1+\tau}{2}) \end{pmatrix}\Bigg]\left(2\pi r^\frac{3}{2}c_{x_1}^\frac{1}{2}\right)^\frac{\sigma_3}{2},
\end{multline*}
\begin{align*}
    &E_{x_1}'(x_1)=e^{i\pi\nu}e^{-\frac{i\pi}{4}\sigma_3}\left[\begin{pmatrix} \mathcal{G}(\frac{\tau}{2}) & \varphi_{x_1}^{(\frac{1}{2})}\mathcal{G}'(\frac{\tau}{2}) \\ 0 & 0 \end{pmatrix}\left(\frac{\beta_{x_1}^{(\frac{3}{4})}}{\mathcal{G}(\frac{1}{2})}I+\frac{c_{x_1}^{(2)}\beta_{x_1}^{(-\frac{1}{4})}}{4\mathcal{G}(\frac{1}{2})}\sigma_3\right)+\right. \\
    &+\frac{\beta_{x_1}^{(-\frac{1}{4})}}{\mathcal{G}(\frac{1}{2})}\begin{pmatrix} \frac{1}{2}(\varphi_{x_1}^{(\frac{1}{2})})^2\mathcal{G}''(\frac{\tau}{2}) & \varphi_{x_1}^{(\frac{3}{2})}\mathcal{G}'(\frac{\tau}{2})+\frac{1}{6}(\varphi_{x_1}^{(\frac{1}{2})})^3\mathcal{G}'''(\frac{\tau}{2}) \\ 0 & 0 \end{pmatrix}+\frac{1}{\beta_{x_1}^{(-\frac{1}{4})}\mathcal{G}(0)}\begin{pmatrix} c_\mathcal{G} & c_\mathcal{G} \\ 1 & 1 \end{pmatrix}\times \\
    &\left.\times\begin{pmatrix} \bigg(\frac{c_{x_1}^{(2)}}{4}-\frac{\beta_{x_1}^{(\frac{3}{4})}}{\beta_{x_1}^{(-\frac{1}{4})}}-\frac{\varphi_{x_1}^{(\frac{3}{2})}}{\varphi_{x_1}^{(\frac{1}{2})}}\bigg)\frac{\tilde{\mathcal{G}}(\frac{1+\tau}{2})}{\varphi_{x_1}^{(\frac{1}{2})}}+\frac{\varphi_{x_1}^{(\frac{1}{2})}}{2}\tilde{\mathcal{G}}''(\frac{1+\tau}{2}) & \hspace{-0.5cm} 0 \\ 0 & \hspace{-0.5cm} \frac{(\varphi_{x_1}^{(\frac{1}{2})})^2}{6}\tilde{\mathcal{G}}'''(\frac{1+\tau}{2})-\bigg(\frac{\beta_{x_1}^{(\frac{3}{4})}}{\beta_{x_1}^{(-\frac{1}{4})}}+\frac{c_{x_1}^{(2)}}{4}\bigg)\tilde{\mathcal{G}}'(\frac{1+\tau}{2}) \end{pmatrix}\right]\left(2\pi r^\frac{3}{2}c_{x_1}^\frac{1}{2}\right)^\frac{\sigma_3}{2},
\end{align*}
where $\widetilde{\mathcal{G}}(z) = \mathcal{G}(z)(z-\frac{1+\tau}{2})$, the coefficients $\beta_{x_1}^{(-\frac{1}{4})}$, $\beta_{x_1}^{(\frac{3}{4})}$, $\varphi_{x_1}^{(\frac{1}{2})}$, $\varphi_{x_1}^{(\frac{3}{2})}$, $c_{x_{1}}^{(2)}$ are defined through the expansions
\begin{align*}
& \beta(z) = \beta_{x_{1}}^{(-\frac{1}{4})}(z-x_{1})^{-\frac{1}{4}} + \beta_{x_{1}}^{(\frac{3}{4})}(z-x_{1})^{\frac{3}{4}} + \bigO((z-x_{1})^{\frac{7}{4}}), & & \text{as }  z\to x_1, \; \im z > 0, \\
& \varphi(z)=\varphi_{x_1}^{(\frac{1}{2})}(z-x_1)^\frac{1}{2}+\varphi_{x_1}^{(\frac{3}{2})}(z-x_1)^\frac{3}{2}+\bigO ((z-x_1)^\frac{5}{2}), & & \text{as }  z\to x_1, \; \im z > 0,
\end{align*}
and via \eqref{cx1}, and they are given by
\begin{align*}
& \beta_{x_{1}}^{(-\frac{1}{4})} = e^{\frac{\pi i }{4}}(x_{0}-x_{1})^{\frac{1}{4}}(x_{0}-x_{2})^{\frac{1}{4}}, \quad \beta_{x_{1}}^{(\frac{3}{4})} = \frac{e^{\frac{\pi i}{4}}(x_{0}-2x_{1}+x_{2})}{4(x_{0}-x_{1})^{\frac{3}{4}}(x_{1}-x_{2})^{\frac{3}{4}}}, \quad \varphi_{x_1}^{(\frac{1}{2})} = \frac{-2i c_{0}}{\sqrt{x_{0}-x_{1}}\sqrt{x_{1}-x_{2}}}, \\
& \varphi_{x_1}^{(\frac{3}{2})} = \frac{ic_{0}(x_{0}-2x_{1}+x_{2})}{3(x_{0}-x_{1})^{\frac{3}{2}}(x_{1}-x_{2})^{\frac{3}{2}}}, \qquad c_{x_{1}}^{(2)} = \frac{2q'(x_{1})}{3q(x_{1})}-\frac{x_{0}-2x_{1}+x_{2}}{3(x_{0}-x_{1})(x_{1}-x_{2})}.
\end{align*}
Each of the expressions for $E_{x_1}(x_1)$, $E_{x_2}(x_2)$, $E_{x_1}'(x_1)$ and $E_{x_2}'(x_2)$ contain the $\theta$-function and its derivatives at the points $0$, $\frac{1}{2}$, $\frac{\tau}{2}$ and $\frac{1+\tau}{2}$. We use $\theta$-function identities to simplify these expressions following the strategy of \cite{BCL19}.

\medskip The function $\varphi$ can be analytically continued to the Riemann surface $X$ defined in the introduction, and its quotient $\varphi_{A}(z):= \varphi(z) \mod (\mathbb{Z}+\tau \mathbb{Z})$ is a bijection from $X$ to $\mathbb{C}/(\mathbb{Z}+\tau \mathbb{Z})$ \cite{FarkasKra}. The following proposition can be taken straight from \cite{BCL19} after replacing $y_{1}$, $y_{2}$ and $z$ with $x_{0}-x_{2}$, $x_{1}-x_{2}$ and $z-x_{2}$, respectively, and noting that $\varphi_{A}(z-x_{2})$ in \cite{BCL19} equals $\varphi_{A}(z)$ here.
\begin{proposition}[Some $\theta$-function identities]\label{prop: Abel map and theta function}
For all $\hat\nu \in X$, we have
\begin{subequations}\label{thetaD123 nu}
\begin{align}
\theta (\hat\nu +\tfrac{\tau}{2})^2
& = e^{-2i \pi \hat\nu}\frac{D_1^2}{\hat{a}-x_2}\theta (\hat\nu)^2, & & D_1 = (x_0-x_2)^\frac{1}{4}(x_1-x_2)^{\frac{1}{4}}e^{- \frac{\pi i \tau}{4}}, \label{thetaD123a nu}
	\\
\theta(\hat\nu +\tfrac{1+\tau}{2})^2
& = e^{-2i \pi \hat\nu}\frac{D_2^2(\hat{a}-x_0)}{\hat{a}-x_2}\theta (\hat\nu)^2, & & D_2 = i \frac{(x_1-x_2)^{\frac{1}{4}}e^{-\frac{\pi i \tau}{4}}}{(x_0-x_1)^{\frac{1}{4}}} , \label{thetaD123b nu}
	\\ 
\theta(\hat\nu +\tfrac{1}{2})^2
& = \frac{D_3^2(\hat{a}-x_1)}{\hat{a}-x_2}\theta (\hat\nu)^2, & & D_3 = \frac{(x_0-x_2)^{\frac{1}{4}}}{(x_0-x_1)^{\frac{1}{4}}}. \label{thetaD123c nu}
\end{align}
\end{subequations}
where $\hat{a} = \varphi_{A}^{-1}(\hat\nu)$. These expressions make it possible to express $\theta^{(j)}( \frac{1}{2} )$, $\theta^{(j)}( \frac{\tau}{2} )$, $\theta^{(j)}( \frac{1+\tau}{2} )$ for $j \geq 0$ in terms of $\theta^{(j)}(0)$, $j \geq 0$. For example, 
\begin{align*}
& \theta\Big(\frac{\tau}{2}\Big)=e^{-\frac{\pi i \tau}{4}}\frac{(x_1-x_2)^{\frac{1}{4}}}{(x_0-x_2)^{\frac{1}{4}}}\theta(0), \qquad \theta\Big(\frac{1}{2}\Big) = \frac{(x_0-x_1)^{\frac{1}{4}}}{(x_0-x_2)^{\frac{1}{4}}}\theta(0), \qquad \theta'\Big(\frac{\tau}{2}\Big) = -i \pi e^{-\frac{\pi i \tau}{4}}\frac{(x_1-x_2)^{\frac{1}{4}}}{(x_0-x_2)^{\frac{1}{4}}}\theta(0),
	\\
& \theta'\Big(\frac{1+\tau}{2}\Big) = i e^{-\frac{\pi i \tau}{4}}\frac{(x_0-x_1)^{\frac{1}{4}}(x_1-x_2)^{\frac{1}{4}}}{2c_{0}}\theta(0), \qquad
\theta''\Big(\frac{1+\tau}{2}\Big) = \pi e^{-\frac{\pi i \tau}{4}} \frac{(x_0-x_1)^{\frac{1}{4}}(x_1-x_2)^{\frac{1}{4}}}{c_{0}}\theta(0),
	\\
& \theta''\Big(\frac{\tau}{2}\Big) = e^{-\frac{\pi i \tau}{4}}\frac{(x_1-x_2)^{\frac{1}{4}}}{(x_0-x_2)^{\frac{1}{4}}}\bigg( \theta''(0)-\bigg[\pi^{2}+\frac{x_0-x_1}{4 c_{0}^{2}} \bigg]\theta(0) \bigg), 	
	\\
& \theta''\Big(\frac{1}{2}\Big) = \frac{(x_0-x_1)^{\frac{1}{4}}}{4 c_{0}^{2} (x_0-x_2)^{\frac{1}{4}}}\Big( (x_1-x_2) \theta(0) + 4 c_{0}^{2} \theta''(0) \Big).
\end{align*}
\end{proposition}

As previously mentioned, we compute $E_{x_j}(x_j)$, $E_{x_j}'(x_j)$, $j=1,2$, via \eqref{Px1}, \eqref{Px2} and now use Proposition \ref{prop: Abel map and theta function} to obtain
\begin{align*}
\frac{-x_1}{2\pi i r} \Big[ E_{x_1}(x_1)^{-1} E_{x_1}'(x_1)\Big]_{21}
&=\frac{-8\pi c_0x_1q(x_1)}{3\Omega(x_0-x_1)(x_1-x_2)}\frac{d}{dr}\left[\log\theta(\nu)\right], \\
\frac{x_2}{2\pi i r} \Big[ E_{x_2}(x_2)^{-1} E_{x_2}'(x_2)\Big]_{21}
 & = \frac{8\pi c_0x_2q(x_2)}{3\Omega(x_0-x_2)(x_1-x_2)}\frac{d}{dr}\left[\log\theta(\nu)\right],
\end{align*}
where $\nu$ is defined in \eqref{nu_r_relation}. It follows that $I_3(r) = c_{3} \, \partial_{r} \log \theta(\nu)$,
where the constant $c_{3}$ is given by
\begin{align*}
c_{3} = \frac{-8\pi c_0}{3\Omega(x_1-x_2)}\left(\frac{x_1q(x_1)}{x_0-x_1}-\frac{x_2q(x_2)}{x_0-x_2}\right)=- \frac{4\pi c_{0}}{3\Omega} (x_0+x_1+x_2).
\end{align*}
Using \eqref{OmegaEllipticInt}, \eqref{def of u}, and the property \cite[Eq. 8.122]{GRtable} of complete elliptic integrals, we find 
\begin{align*}
\frac{\Omega}{c_0}&=\frac{8}{3}(x_0+x_1+x_2)\left[\textbf{K}(k')\left(\textbf{K}(k)-\textbf{E}(k)\right)-\textbf{E}(k')\textbf{K}(k)\right]=-\frac{4\pi}{3}(x_0+x_1+x_2),
\end{align*}
which shows that $c_{3}=1$. This proves that
\begin{align}\label{third term simple sum and c3}
I_3(r) = \partial_{r} \log \theta(\nu).
\end{align}

\subsection{Analysis of $I_2(r)$}
By \eqref{sandwich in terms of Jcal}, we have
\begin{align}\label{secontermexpression}
 I_2(r) = &\; \frac{1}{2\pi i r^{\frac{5}{2}}}\bigg\{\sum_{j=1}^2(-1)^{j+1}x_j(\mathcal{J}_{x_j})_{x_j}^{(1)}  
 + \sum_{j=1}^2\frac{2(-1)^jx_j(\mathcal{J}_{x_j})_{x_0}^{(-2)}}{(x_0-x_j)^3}
  + \sum_{j=1}^2\sum_{\substack{j'=0 \\ j'\neq j}}^2\frac{(-1)^{j+1}x_j(\mathcal{J}_{x_{j}})_{x_{j'}}^{(-1)}}{(x_{j}-x_{j'})^{2}}\bigg\}.
\end{align}
To simplify the right-hand side, we need Proposition \ref{prop: Abel map and theta function} as well as the following identities.

\begin{proposition}\label{theta'/theta derivative identity}
We have
\begin{align}
& \frac{\theta''(0)}{\theta(0)}=-\frac{(x_1-x_2)(\gamma x_0 + \delta)}{4c_0^2(\gamma x_2 + \delta)}=\frac{(x_0-x_1)(3x_0+x_1-x_2)}{4c_0^2(x_0+x_1+x_2)}, \label{theta''/theta} \\
& \frac{d}{dz}\left[\frac{\theta'(\varphi(z))}{\theta(\varphi(z))}\right]= \frac{\gamma z + \delta}{z-x_2}\varphi'(z), \qquad z \in X, \label{rationalFunc}
\end{align}
where 
$$\gamma = \frac{(x_0-x_2)(3x_0-x_1+x_2)}{4c_0^2(x_0+x_1+x_2)}, \qquad
\delta = -\frac{(x_0-x_2) (x_0 (x_1+2 x_2)+x_1 (x_1-x_2))}{4 c_0^2 (x_0+x_1+x_2)}.$$
\end{proposition}
\begin{proof}
We will first show that there exist $\gamma, \delta \in \C$ such that
\begin{align}\label{identity with unknowns}
\frac{\theta''(\varphi(z))\theta(\varphi(z))-\theta'(\varphi(z))^2}{\theta(\varphi(z))^2}=\frac{\gamma z + \delta}{z-x_2}.
\end{align}

Let $h(z)$ denote the left-hand side of (\ref{identity with unknowns}).
Using \eqref{prop of theta genus g} and the fact that $\frac{1+\tau}{2}$ is a simple zero of $\theta(u)$, we verify that $\big(\frac{\theta'(u)}{\theta(u)}\big)'=\frac{\theta''(u)\theta(u)-\theta'(u)^2}{\theta(u)^2}$ is well-defined on $\mathbb{C}/(\mathbb{Z}+\tau \mathbb{Z})$ with a double pole at $\frac{1+\tau}{2}$ and no other poles. Since the Abel map $\varphi_A:X \to \mathbb{C}/(\mathbb{Z}+\tau \mathbb{Z})$ is an isomorphism, it follows that $h$ is a meromorphic function on $X$ with a double pole at $\varphi_A^{-1}(\frac{1+\tau}{2}) = x_2$ and no other poles. In fact, since $\theta(u)$ is an even function of $u$ and $\varphi(z) = -\varphi(\iota(z))$, where $\iota:X \to X$ denotes the sheet-changing involution, $h$ descends to a meromorphic function on the Riemann sphere with a simple pole at $x_2$ and no other poles. This shows that (\ref{identity with unknowns}), and hence also (\ref{rationalFunc}), holds for some choice of $\gamma, \delta \in \C$.

Taking $z\to\infty$ in (\ref{identity with unknowns}) and using Proposition \ref{prop: Abel map and theta function}, we see that
\begin{align}\label{first expr for c tilde}
    \gamma=\frac{\theta''(\frac{1}{2})}{\theta(\frac{1}{2})}=\frac{x_1-x_2}{4c_0^2}+\frac{\theta''(0)}{\theta(0)}.
\end{align}
Multiplying \eqref{identity with unknowns} by $\varphi'(z)$, integrating from $x_0$ to $x>x_{0}$, and using \eqref{def of u} and \cite[Eqs. 3.131.7, 3.133.17]{GRtable}, we obtain 
\begin{align}\label{thetaprimeovertheta}
\frac{\theta'(\varphi(x))}{\theta(\varphi(x))}&= \frac{\gamma x_1+ \delta}{x_1-x_2}\varphi(x)+\frac{2(\gamma x_2 +\delta)(x-x_0)}{x_0-x_2}\varphi'(x)-\frac{\gamma x_2  + \delta}{2(x_1-x_2)}\frac{\textbf{E}\left(\sin^{-1}\sqrt{\frac{x-x_0}{x-x_1}},k\right)}{\textbf{K}\left(k\right)},    
\end{align}
where $k=\sqrt{\frac{x_1-x_2}{x_0-x_2}}$, and $\textbf{E}(\varphi,k):=\int_0^\varphi \sqrt{1-k^2\sin^2t}~dt$ is the elliptic integral of the second kind.  Sending $x\to\infty$ in (\ref{thetaprimeovertheta}) and using \eqref{x0 implicit eq}, we obtain $\delta$ in terms of $\gamma$. Then, taking $z\to x_0$ in \eqref{identity with unknowns}, we find \eqref{theta''/theta}. Substituting \eqref{theta''/theta} into (\ref{first expr for c tilde}), we find the asserted expression for $\gamma$ and as an immediate consequence we find the asserted expression for $\delta$.
\end{proof}

The quantity $a:=\varphi_{A}^{-1}(\nu) \in X$ appears in the large $r$ asymptotics of \eqref{secontermexpression}. From \eqref{def of u}, we note that $\varphi(x)$ is monotone for $x \in (x_0,+\infty)$, and satisfies $\varphi_{A}(x_0) = 0$ and $\varphi_{A}(\infty) = \frac{1}{2}$. Hence, as $r$ increases, $a$ oscillates between $x_{0}$ and $+\infty$ (and changes sheet). More precisely, $a$ belongs to the upper sheet and satisfies $\sqrt{\mathcal{R}(a)} > 0$ if $\nu \hspace{-.1cm}\mod 1 \in (0,\frac{1}{2})$, while $a$ belongs to the lower sheet and satisfies $\sqrt{\mathcal{R}(a)} < 0$ if $\nu \mod 1 \in (\frac{1}{2},1)$.

\begin{proposition}\label{prop: second term exact}
We have
\begin{align}
I_2(r) = \frac{d}{dr}\bigg\{& \tilde{c}_0\log(r)
+\int^r\bigg[ \frac{\tilde{c}_{-2}}{\tilde{r}(\tilde{a}-x_2)^2}+\frac{\tilde{c}_{-1}}{\tilde{r}(\tilde{a}-x_2)}+\frac{\tilde{c}_1(\tilde{a})}{\tilde{r}}\frac{\theta'(\tilde{\nu})}{\theta(\tilde{\nu})} \nonumber
	\\ \label{I2expression}
 &+\frac{\tilde{c}_2(\tilde{a})}{\tilde{r}^\frac{3}{2}}\frac{d}{d\tilde{r}}\bigg(\frac{\theta'(\tilde{\nu})}{\theta(\tilde{\nu})}\bigg) +\frac{\tilde{c}_3}{\tilde{r}^\frac{3}{2}}\frac{d}{d\tilde{r}}\bigg(\frac{\theta'(\tilde{\nu})^3}{\theta(\tilde{\nu})^3}\bigg)\bigg]d\tilde{r}\bigg\}
\end{align}
where $\tilde{\nu}=-\frac{\Omega}{2\pi}\tilde{r}^\frac{3}{2}$, $\tilde{a}=\varphi^{-1}_{A}(\tilde{\nu})$, and
\begin{align*}
    \tilde{c}_{-2}= &\; \frac{17(x_0-x_2)(x_1-x_2)\left(x_1\tilde{q}(x_1)-x_2\tilde{q}(x_2)\right)}{48\tilde{q}(x_0)(x_0-x_1)}, 
    	\\
    \tilde{c}_{-1}= &\; \frac{1}{48}\bigg\{\frac{3x_2q(x_2)}{q(x_1)}-\frac{3 x_1 \tilde{q}(x_1)}{\tilde{q}(x_2)}+\frac{x_1 \tilde{q}(x_1)(-29 x_0^2-20 x_0 x_1+21 x_0 x_2-11 x_1^2+9 x_1 x_2+30 x_2^2)}{\tilde{q}(x_0) (x_0-x_1) (x_0+x_1+x_2)}
    	\\
    & +\frac{x_2 \tilde{q}(x_2) \left(29 x_0^2+13 x_0 x_1-14 x_0 x_2+4 x_1^2-9 x_1 x_2-23 x_2^2\right)}{\tilde{q}(x_0) (x_0-x_1) (x_0+x_1+x_2)}+\frac{3 (x_0-x_2) (x_1 \tilde{q}(x_1)-x_2 \tilde{q}(x_2))}{\tilde{q}(x_0)^2}\bigg\}, 	    
\end{align*}
\begin{align*}
\tilde{c}_{0}=&\; \frac{1}{48}\bigg\{\frac{2 x_2 \tilde{q}(x_2) \left(-11 x_0^2+2 x_0 (x_1+x_2)+3 x_1^2+x_1 x_2+3 x_2^2\right)+7 x_1 \tilde{q}(x_1) (x_0-x_2) (x_0+x_1+x_2)}{\tilde{q}(x_0) (x_0-x_1) (x_0-x_2) (x_0+x_1+x_2)}  
    	\\
    &+ \frac{3x_2\tilde{q}(x_2)}{\tilde{q}(x_0)^2}+\frac{3x_1\tilde{q}(x_1)}{\tilde{q}(x_2) (x_1-x_2)}-\frac{3x_2q(x_2)}{q(x_1)(x_1-x_2)}+3 x_0 \left(\frac{6}{x_0+x_1+x_2}+\frac{1}{x_0-x_1}+\frac{1}{x_0-x_2}\right) 
    	\\
    &-\frac{3x_1q'(x_1)}{q(x_1)}-\frac{3 x_2 q'(x_2)}{q(x_2)}\bigg\}, \\
    \tilde{c}_{1}(\tilde{a})=&\; \frac{c_0\left(x_1\tilde{q}(x_1)-x_2\tilde{q}(x_2)\right)}{\tilde{q}(x_0)(x_0-x_1)}\frac{\sqrt{\mathcal{R}(\tilde{a})}}{(\tilde{a}-x_2)^2}, 
    	\\
    \tilde{c}_{2}(\tilde{a})= &\; \frac{-\pi c_0^2}{9\Omega\tilde{q}(x_0)(x_0-x_1)}\left(\frac{5(x_1\tilde{q}(x_1)-x_2\tilde{q}(x_2))}{\tilde{a}-x_2}+\frac{5x_2\tilde{q}(x_2)-9x_0\tilde{q}(x_0)}{x_0-x_2}\right), 
    	\\
    \tilde{c}_{3}=&\; \frac{-8\pi c_0^4(x_0+x_1+x_2)}{9\Omega\tilde{q}(x_0)(x_0-x_1)(x_0-x_2)}.
\end{align*}
\end{proposition}
\begin{proof}
First, using \eqref{explicit expression for the jumps Jrp1p}, we compute the coefficients in the expansion of $J_{R}^{(1)}(z)$ as $z \to x_j$, $j = 1,2$. By combining these coefficients with the expressions for $E_{x_{j}}(x_{j})$ (already needed in Subsection \ref{subsection: third term}), this gives us expressions for the quantities $(\mathcal{J}_{x_{j}})_{x_{j}}^{(1)}$, $(\mathcal{J}_{x_{j}})_{x_0}^{(-2)}$ and $(\mathcal{J}_{x_{j}})_{x_{j'}}^{(-1)}$, for $j=1,2$ and $j'=0,1,2$, $j \neq j'$. These expressions contain the $\theta$-function and its derivatives evaluated at the eight points $0$, $\frac{1}{2}$, $\frac{\tau}{2}$, $\frac{1+\tau}{2}$, $\nu$, $\nu+\frac{1}{2}$, $\nu+\frac{\tau}{2}$ and $\nu+\frac{1+\tau}{2}$. Second, we use Proposition \ref{prop: Abel map and theta function} to express the quantities $\theta^{(j)}(\frac{1}{2})$, $\theta^{(j)}(\frac{\tau}{2})$, $\theta^{(j)}(\frac{1+\tau}{2})$, $\theta^{(j)}(\nu+\frac{1}{2})$, $\theta^{(j)}(\nu+\frac{\tau}{2})$ and $\theta^{(j)}(\nu+\frac{1+\tau}{2})$ in terms of only $a$, $\theta^{(j)}(0)$ and $\theta^{(j)}(\nu)$, $j\geq 0$. This allows us to obtain the following simplified expressions:
\begin{align*}
    (\mathcal{J}_{x_1})_{x_1}^{(1)}&=\frac{i\pi r^\frac{3}{2}}{-8}\Bigg[\frac{q'(x_1)}{q(x_1)}+\frac{4c_0^2\left(\frac{\theta'(\nu)}{\theta(\nu)}\right)^2+12c_0^2\left(\frac{\theta''(\nu)}{\theta(\nu)}-\frac{\theta''(0)}{\theta(0)}\right)+x_0-2x_1+x_2}{(x_0-x_1)(x_1-x_2)}\Bigg], \\
    (\mathcal{J}_{x_2})_{x_2}^{(1)}&=\frac{i\pi r^\frac{3}{2}}{8}\Bigg[\frac{q'(x_2)}{q(x_2)}-\frac{4c_0^2\left(\frac{\theta'(\nu)}{\theta(\nu)}\right)^2+12c_0^2\left(\frac{\theta''(\nu)}{\theta(\nu)}-\frac{\theta''(0)}{\theta(0)}\right)+x_0-2x_1+x_2}{(x_0-x_1)(x_1-x_2)}\Bigg],
\end{align*}
\begin{align*}
    &\left(\mathcal{J}_{x_1}\right)_{x_0}^{(-2)}=\frac{5\pi ir^\frac{3}{2}(x_0-x_1)(x_0-x_2)q(x_1)}{-24(a-x_2)\tilde{q}(x_0)}, & & \left(\mathcal{J}_{x_2}\right)_{x_0}^{(-2)}=\frac{5\pi ir^\frac{3}{2}(a-x_0)(x_0-x_2)q(x_2)}{24(a-x_2)\tilde{q}(x_0)},
\end{align*}
\begin{align*}
    &\left(\mathcal{J}_{x_2}\right)_{x_1}^{(-1)}=\frac{i\pi r^{\frac{3}{2}}(x_1-x_2)q(x_2)}{8q(x_1)}\left(1-\frac{x_1-x_2}{a-x_2}\right), & & \left(\mathcal{J}_{x_1}\right)_{x_2}^{(-1)}=\frac{i\pi r^\frac{3}{2}(x_1-x_2)\tilde{q}(x_1)}{8\tilde{q}(x_2)}\left(1-\frac{x_1-x_2}{a-x_2}\right).
\end{align*}
The expressions for $(\mathcal{J}_{x_j})_{x_0}^{(-1)}$, $j=1,2$, are significantly larger and also contain the quantities $\big(\frac{\theta'(\nu)}{\theta(\nu)}\big)^2$, $\frac{\theta''(\nu)}{\theta(\nu)}$, $\frac{\theta''(0)}{\theta(0)}$ and $a$. In addition, they contain terms involving $\sqrt{\mathcal{R}(a)}\frac{\theta'(\nu)}{\theta(\nu)}$. Third, we use \eqref{theta''/theta} to express $\frac{\theta''(0)}{\theta(0)}$ in terms of the $x_{j}$, and we also use the relation $\frac{\theta''(\nu)}{\theta(\nu)} = \big(\frac{\theta'(\nu)}{\theta(\nu)}\big)' + \frac{\theta'(\nu)^2}{\theta(\nu)^2}$. Collecting all of the terms involving $\frac{\theta'(\nu)^2}{\theta(\nu)^2}$, we obtain
\begin{align*}
\frac{3\tilde{c}_{3}}{r^{\frac{3}{2}}}\frac{\theta'(\nu)^2}{\theta(\nu)^2}
\frac{\gamma a + \delta}{a-x_2} \partial_{r}\nu,
\end{align*}
which can be rewritten as $\frac{\tilde{c}_3}{r^\frac{3}{2}}\frac{d}{dr}\left[\frac{\theta'(\nu)^3}{\theta(\nu)^3}\right]$ by using \eqref{rationalFunc}.
\end{proof}

The next proposition generalizes \cite[Proposition 8.6]{BCL19}.

\begin{proposition}\label{prop: 1periodic log int}
Let $M>0$ and let $\mathcal{H} \in L^1([0,1])$ be a periodic function of period $1$ satisfying $\mathcal{H}(-u) = \mathcal{H}(u)$ for all $u \in [0,1]$. As $r\to+\infty$,
\begin{align*}
\int_M^r \mathcal{H} (\tilde{\nu} )\frac{d\tilde{r}}{\tilde{r}}=\log(r)\int_{0}^{1}\mathcal{H}(s)ds+\tilde{C}+\bigO(r^{-\frac{3}{2}}),
\end{align*}
where $\tilde{\nu} = -\frac{\Omega}{2\pi}\tilde{r}^\frac{3}{2}$ and $\tilde{C}$ is a constant independent of $r$. In particular, there are no oscillations of order $1$.
\end{proposition}
\begin{proof}
Recall that $\Omega > 0$. Hence, 
\begin{align*}
\int_M^r  \mathcal{H}\left(-\frac{\Omega}{2\pi}\tilde{r}^\frac{3}{2}\right)\frac{d\tilde{r}}{\tilde{r}}=\frac{2}{3}\int_{M'}^\nu \mathcal{H}(\tilde{\nu})\frac{d\tilde{\nu}}{\tilde{\nu}}=\frac{2}{3}\left(\int_{M'}^{n_0}+\int_{n_r}^\nu\right)\mathcal{H}(\tilde{\nu})\frac{d\tilde{\nu}}{\tilde{\nu}}+\frac{2}{3}\sum_{j=n_0}^{n_r-1}\int_0^{-1}\frac{\mathcal{H}(s)}{s+j}ds,
\end{align*}
where $M'=-\frac{\Omega}{2\pi}M^\frac{3}{2}<0$, $n_0=\left\lfloor{M'}\right\rfloor$ and $n_r=\left\lceil{\nu}\right\rceil$. Because $\mathcal{H}\in L^1([0,1])$ and $j \leq n_{0}<0$, we have
\begin{align*}
& \frac{2}{3}\left(\int_{M'}^{n_0}+\int_{n_r}^\nu\right)\mathcal{H}(\tilde{\nu})\frac{d\tilde{\nu}}{\tilde{\nu}}=\tilde{C}_{1}+\bigO(r^{-\frac{3}{2}}), \\
& \sum_{j=n_0}^{n_r-1} \int_0^{-1}\frac{\mathcal{H}(s)}{s+j}ds= \sum_{j=n_0}^{n_r-1}\frac{1}{j}\int_0^{-1} \mathcal{H}(s) ds+\tilde{C}_{2}+\bigO(r^{-3}),
\end{align*}
as $r \to + \infty$, for certain constants $\tilde{C}_{1}$, $\tilde{C}_{2}$. Finally, since $\mathcal{H}(-s)=\mathcal{H}(s)$ and $n_{r}<0$, we deduce that
\begin{align*}
\int_0^{-1} \mathcal{H}(s)ds\sum_{j=n_0}^{n_r-1}\frac{1}{j}=\frac{3}{2} \log r \int_0^1 \mathcal{H}(s)ds+\tilde{C}_{3} + \bigO(r^{-\frac{3}{2}}), \qquad \mbox{as } r \to + \infty,
\end{align*}
for some constant $\tilde{C}_{3}$.
\end{proof}

We next extract the leading order behavior from Proposition \ref{prop: second term exact}.
\begin{proposition}\label{prop: c2 not simplified}
Fix $M>0$. As $r\to+\infty$,
\begin{align}\label{secontermexpression 2}
    \int_M^r I_2(\tilde{r}) d\tilde{r}=c_2\log(r)+\tilde{C}_4+\bigO ( r^{-\frac{3}{2}} ),
\end{align}
where $\tilde{C}_4$ is independent of $r$ and
\begin{align}\label{c2}
    c_2= &\; \tilde{c}_0+\frac{\tilde{c}_{-1}(3x_0-x_1+x_2)}{(x_1-x_2)(x_0+x_1+x_2)}+\tilde{c}_{-2}\frac{2x_0^2+x_0(x_1-3x_2)-(x_1-x_2)^2}{(x_0-x_2)(x_1-x_2)^2(x_0+x_1+x_2)} \nonumber \\
    & -\frac{7}{6}\int_{x_0}^\infty \tilde{c}_1(x)\frac{\theta'(\varphi(x))}{\theta(\varphi(x))}\varphi'(x)dx,
\end{align}
with the integration contour from $x_0$ to $\infty$ taken on the upper sheet of $X$.
\end{proposition}
\begin{proof}
The first term on the right-hand side of \eqref{c2} follows trivially from (\ref{I2expression}). 
Since $u \mapsto \varphi_{A}^{-1}(u)$ is even and periodic of period $1$, we can apply Proposition \ref{prop: 1periodic log int} with $\mathcal{H}(\tilde{\nu})=\frac{\tilde{c}_{-2}}{(\varphi_{A}^{-1}(\tilde{\nu})-x_2)^2}+\frac{\tilde{c}_1}{(\varphi_{A}^{-1}(\tilde{\nu})-x_2)}$. Moreover, with the help of \cite[Eqs. 3.133.18, 3.134.18]{GRtable}, \eqref{x0 implicit eq}, and \eqref{def of u}, we compute
\begin{subequations}\label{dsintegrals}
\begin{align}
    \int_0^1\frac{ds}{\varphi_{A}^{-1}(s)-x_2}&=2\int_{x_0}^\infty\frac{\varphi'(\xi) d\xi}{\xi-x_2}=\frac{3x_0-x_1+x_2}{(x_1-x_2)(x_0+x_1+x_2)}, 
    	\\
    \int_0^1\frac{ds}{(\varphi_{A}^{-1}(s)-x_2)^2}&=2\int_{x_0}^\infty\frac{\varphi'(\xi) d\xi}{(\xi-x_2)^2}=\frac{2x_0^2+x_0(x_1-3x_2)-(x_1-x_2)^2}{(x_0-x_2)(x_1-x_2)^2(x_0+x_1+x_2)},
\end{align}
\end{subequations}
where the integration contours in the integrals with respect to $d\xi$ lie on the upper sheet.
Using these identities to evaluate $\int_0^1 \mathcal{H}(s)ds$ explicitly, we obtain the second and third terms on the right-hand side of \eqref{c2}.
Next, we use Proposition \ref{prop: 1periodic log int} with $\mathcal{H}(\tilde{\nu}) = \tilde{c}_1(\tilde{a}) \frac{\theta'(\tilde{\nu})}{\theta(\tilde{\nu})}$ to obtain
\begin{align}\label{term 4 first part}
    \int_M^r \frac{\tilde{c}_1(\tilde{a})}{\tilde{r}} \frac{\theta'(\tilde{\nu})}{\theta(\tilde{\nu})} d\tilde{r}=-2\log(r)\int_{x_0}^\infty\tilde{c}_1(x)\frac{\theta'(\varphi(x))}{\theta(\varphi(x))}\varphi'(x)dx+\tilde{C}_{5} + \bigO(r^{-\frac{3}{2}}),
\end{align}
where the contour in the $x$-integral lies on the upper sheet. Furthermore, integrating by parts and then using Proposition \ref{prop: 1periodic log int} again, we find
\begin{align}
    \int_M^r\frac{\tilde{c}_2(a)}{\tilde{r}^\frac{3}{2}}\frac{d}{d\tilde{r}}\left[\frac{\theta'(\tilde{\nu})}{\theta(\tilde{\nu})}\right]d\tilde{r}&=-\frac{5}{12}\int_M^r\frac{\tilde{c}_1(a)}{\tilde{r}} \frac{\theta'(\tilde{\nu})}{\theta(\tilde{\nu})} d\tilde{r}+\frac{3}{2}\int_M^r\frac{\tilde{c}_2(a)}{\tilde{r}^\frac{5}{2}}\frac{\theta'(\tilde{\nu})}{\theta(\tilde{\nu})}d\tilde{r} \nonumber \\
    &=\frac{5}{6}\log(r)\int_{x_0}^\infty\tilde{c}_1(x)\frac{\theta'(\varphi(x))}{\theta(\varphi(x))}\varphi'(x)dx+\tilde{C}_{6} + \bigO ( r^{-\frac{3}{2}} ), \label{term 4 second part}
\end{align}
where again the contour in the $x$-integral lies on the upper sheet.
We obtain the last term on the right-hand side of \eqref{c2} by adding \eqref{term 4 first part} and \eqref{term 4 second part}.  Lastly, an integration by parts shows that
\begin{align*}
    \int_M^r\frac{\tilde{c}_3}{\tilde{r}^\frac{3}{2}}\frac{d}{d\tilde{r}}\left[\frac{\theta'(\tilde{\nu})^3}{\theta(\tilde{\nu})^3}\right]d\tilde{r}=\tilde{C}_{7} + \bigO(r^{-\frac{3}{2}}),
\end{align*}
which completes the proof.
\end{proof}

We finally simplify the constant $c_2$.
\begin{proposition}\label{prop: c2 is -3/8}
We have $c_{2} = -\frac{3}{8}$.
\end{proposition}
\begin{proof}
We start by evaluating the integral in \eqref{c2}, which can be written as
\begin{align*}
    \int_{x_0}^\infty\tilde{c}_1(x)\frac{\theta'(\varphi(x))}{\theta(\varphi(x))}\varphi'(x)dx=\frac{c_0^2(x_1\tilde{q}(x_1)-x_2\tilde{q}(x_2))}{\tilde{q}(x_0)(x_0-x_1)}\int_{x_0}^\infty\frac{\theta'(\varphi(x))}{\theta(\varphi(x))}\frac{dx}{(x-x_2)^2}.
\end{align*}
Using Proposition \ref{theta'/theta derivative identity}, an integration by parts gives
\begin{align*}
    \int_{x_0}^\infty\frac{\theta'(\varphi(x))}{\theta(\varphi(x))}\frac{dx}{(x-x_2)^2}&= \int_{x_0}^\infty\frac{\gamma x + \delta}{\sqrt{\mathcal{R}(x)}}\frac{dx}{(x-x_2)^2}.
\end{align*}
Computing the integral on the right-hand side as in (\ref{dsintegrals}), we infer that
\begin{align*}
    \int_{x_0}^\infty\tilde{c}_1(x)\frac{\theta'(\varphi(x))}{\theta(\varphi(x))}\varphi'(x)dx=\frac{x_0(7x_0-2x_1-2x_2)-(x_1-x_2)^2}{8(x_0+x_1+x_2)(3x_0-x_1-x_2)}
\end{align*}
and substituting this into \eqref{c2}, we obtain the result.
\end{proof}

Theorem \ref{thm: main result} follows by combining \eqref{partialrlogdet}, (\ref{I1final}), \eqref{third term simple sum and c3}, \eqref{secontermexpression 2}, and Proposition \ref{prop: c2 is -3/8}.

\appendix

\section{Airy model RH problem}\label{subsec:Airyparametrix}
The Airy model RH problem consists of finding a function $\Phi_{\mathrm{Ai}}$ satisfying the following properties.
\begin{itemize}
\item[(a)] $\Phi_{\mathrm{Ai}} : \mathbb{C} \setminus \Sigma_{\mathrm{Ai}} \rightarrow \mathbb{C}^{2 \times 2}$ is analytic, where $\Sigma_{\mathrm{Ai}}$ is shown in Figure \ref{figAiry}.
\item[(b)] $\Phi_{\mathrm{Ai}}$ has the jump relations
\begin{equation}\label{jumps P3}
\begin{array}{l l}
\Phi_{\mathrm{Ai},+}(z) = \Phi_{\mathrm{Ai},-}(z) \begin{pmatrix}
0 & 1 \\ -1 & 0
\end{pmatrix}, & \mbox{ on } \mathbb{R}^{-}, \\

\Phi_{\mathrm{Ai},+}(z) = \Phi_{\mathrm{Ai},-}(z) \begin{pmatrix}
 1 & 1 \\
 0 & 1
\end{pmatrix}, & \mbox{ on } \mathbb{R}^{+}, \\

\Phi_{\mathrm{Ai},+}(z) = \Phi_{\mathrm{Ai},-}(z) \begin{pmatrix}
 1 & 0  \\ 1 & 1
\end{pmatrix}, & \mbox{ on } e^{ \frac{2\pi i}{3} }  \mathbb{R}^{+} , \\

\Phi_{\mathrm{Ai},+}(z) = \Phi_{\mathrm{Ai},-}(z) \begin{pmatrix}
 1 & 0  \\ 1 & 1
\end{pmatrix}, & \mbox{ on }e^{ -\frac{2\pi i}{3} }\mathbb{R}^{+} . \\
\end{array}
\end{equation}
\item[(c)] As $z \to \infty$, $z \notin \Sigma_{\mathrm{Ai}}$, we have
\begin{equation}\label{Asymptotics Airy}
\Phi_{\mathrm{Ai}}(z) = z^{-\frac{\sigma_{3}}{4}}M \left( I + \sum_{k=1}^{\infty} \frac{\Phi_{\mathrm{Ai,k}}}{z^{3k/2}} \right) e^{-\frac{2}{3}z^{\frac{3}{2}}\sigma_{3}},
\end{equation}
where $M = \frac{1}{\sqrt{2}}\begin{pmatrix}
1 & i \\ i & 1
\end{pmatrix}$ and $\Phi_{\mathrm{Ai,1}} = \frac{1}{8}\begin{pmatrix}
\frac{1}{6} & i \\ i & -\frac{1}{6}
\end{pmatrix}$.

\item[(d)] As $z \to 0$, we have $\Phi_{\mathrm{Ai}}(z) = \bigO(1)$.
\end{itemize}
The unique solution $\Phi_{\mathrm{Ai}}$ can be written in terms of Airy functions \cite{DKMVZ1}, but its exact expression is unimportant for us. The quantity we need is the explicit expression for $\Phi_{\mathrm{Ai},1}$.

\begin{figure}[t]
    \begin{center}
    \setlength{\unitlength}{1truemm}
    \begin{picture}(100,55)(0,0)
        \put(50,40){\line(1,0){30}}
        \put(50,40){\line(-1,0){30}}
        \put(50,39.8){\thicklines\circle*{1.2}}
        \put(50,40){\line(-0.5,0.866){10}}
        \put(50,40){\line(-0.5,-0.866){10}}
        \qbezier(53,40)(52,43)(48.5,42.598)
        \put(53,43){$\frac{2\pi}{3}$}
        \put(50.3,36.8){$0$}
        \put(65,39.9){\thicklines\vector(1,0){.0001}}
        \put(35,39.9){\thicklines\vector(1,0){.0001}}
        \put(44.35,49.588){\thicklines\vector(0.5,-0.866){.0001}}
        \put(44.35,30.412){\thicklines\vector(0.5,0.866){.0001}}
    \end{picture}
    \vspace{-2.5cm}\caption{\label{figAiry}The jump contour $\Sigma_{\mathrm{Ai}}$.}
\end{center}
\end{figure}
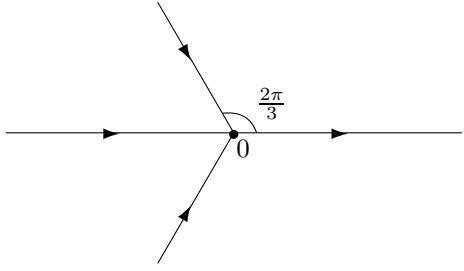

\section{Bessel model RH problem}\label{subsec:Besselparametrix}
The Bessel model RH problem consists of finding a function $\Phi_{\mathrm{Be}}$ satisfying the following properties.
\begin{itemize}
\item[(a)] $\Phi_{\mathrm{Be}} : \mathbb{C} \setminus \Sigma_{\mathrm{Be}} \to \mathbb{C}^{2\times 2}$ is analytic, where
$\Sigma_{\mathrm{Be}}$ is shown in Figure \ref{figBessel}.
\item[(b)] $\Phi_{\mathrm{Be}}$ satisfies the jump conditions
\begin{equation}\label{Jump for P_Be}
\begin{array}{l l} 
\Phi_{\mathrm{Be},+}(z) = \Phi_{\mathrm{Be},-}(z) \begin{pmatrix}
0 & 1 \\ -1 & 0
\end{pmatrix}, & z \in \mathbb{R}^{-}, \\

\Phi_{\mathrm{Be},+}(z) = \Phi_{\mathrm{Be},-}(z) \begin{pmatrix}
1 & 0 \\ 1 & 1
\end{pmatrix}, & z \in e^{ \frac{2\pi i}{3} }  \mathbb{R}^{+}, \\

\Phi_{\mathrm{Be},+}(z) = \Phi_{\mathrm{Be},-}(z) \begin{pmatrix}
1 & 0 \\ 1 & 1
\end{pmatrix}, & z \in e^{ -\frac{2\pi i}{3} }  \mathbb{R}^{+}. \\
\end{array}
\end{equation}
\item[(c)] As $z \to \infty$, $z \notin \Sigma_{\mathrm{Be}}$, we have
\begin{equation}\label{large z asymptotics Bessel}
\Phi_{\mathrm{Be}}(z) = ( 2\pi z^{\frac{1}{2}} )^{-\frac{\sigma_{3}}{2}}M
\left(I+ \frac{\Phi_{\mathrm{Be},1}}{z^{\frac{1}{2}}} + \bigO(z^{-1}) \right) e^{2 z^{\frac{1}{2}}\sigma_{3}},
\end{equation}
where $\Phi_{\mathrm{Be},1} = \frac{1}{16}\begin{pmatrix}
-1 & -2i \\ -2i & 1
\end{pmatrix}$ and $M = \frac{1}{\sqrt{2}}\begin{pmatrix}
1 & i \\ i & 1
\end{pmatrix}$.
\item[(d)] As $z \to 0$, we have
\begin{equation}\label{local behavior near 0 of P_Be}
\Phi_{\mathrm{Be}}(z) = \left\{ \begin{array}{l l}
\begin{pmatrix}
\bigO(1) & \bigO(|\log z|) \\
\bigO(z) & \bigO(|\log z|) 
\end{pmatrix}, & |\arg z| < \frac{2\pi}{3}, \\
\begin{pmatrix}
\bigO(|\log z|) & \bigO(|\log z|) \\
\bigO(|\log z|) & \bigO(|\log z|) 
\end{pmatrix}, & \frac{2\pi}{3}< |\arg z| < \pi.
\end{array}  \right.
\end{equation}
\end{itemize}
\begin{figure}[t]
    \begin{center}
    \setlength{\unitlength}{1truemm}
    \begin{picture}(100,55)(0,0)
        \put(50,40){\line(-1,0){30}}
        \put(50,39.8){\thicklines\circle*{1.2}}
        \put(50,40){\line(-0.5,0.866){10}}
        \put(50,40){\line(-0.5,-0.866){10}}
        \put(50.3,36.8){$0$}
        \put(35,39.9){\thicklines\vector(1,0){.0001}}
        \put(44.35,49.588){\thicklines\vector(0.5,-0.866){.0001}}
        \put(44.35,30.412){\thicklines\vector(0.5,0.866){.0001}}
    \end{picture}
    \vspace{-2.5cm}\caption{\label{figBessel}The jump contour $\Sigma_{\mathrm{Be}}$.}
\end{center}
\end{figure}
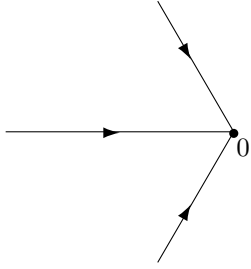
The unique solution $\Phi_{\mathrm{Be}}$ is given in terms of Bessel functions \cite{DIZ}, but its exact expression is unimportant for our purposes. Again, the quantity we need is the explicit expression for $\Phi_{\mathrm{Be},1}$.


\end{document}